\documentclass[12pt]{article}
\usepackage{a4wide}
\usepackage{amsthm}
\usepackage{amsfonts}
\usepackage{amssymb}
\usepackage{amsmath}
\usepackage{cite}
\usepackage{epsfig}

\newtheorem{theorem}{Theorem}
\newtheorem{corollary}[theorem]{Corollary}
\newtheorem{lemma}[theorem]{Lemma}
\newtheorem{proposition}[theorem]{Proposition}
\newtheorem{conjecture}{Conjecture}

\def\CC{{\mathcal C}}
\def\JJ{{\mathcal J}}

\def\dd{\;\mbox{d}}
\def\NN{{\mathbb N}}
\def\RR{{\mathbb R}}
\def\deg{\mbox{deg}}
\def\Tr{\mbox{Tr}\;}

\newcommand\cut[1]{||#1||_{\square}}

\begin{document}

\title{Finitely forcible graph limits are universal\thanks{
This work has received funding from the European Research Council (ERC) under the European Union’s Horizon 2020 research and innovation programme (grant agreement No 648509). This publication reflects only its authors' view; the European Research Council Executive Agency is not responsible for any use that may be made of the information it contains.
In addition,
the first and the second authors were supported by the Leverhulme Trust 2014 Philip Leverhulme Prize,
the second author by the Engineering and Physical Sciences Research Council Standard Grant number EP/M025365/1, and
the third by the CNPq Science Without Borders grant number 200932/2014-4.}}

\author{Jacob W.~Cooper\thanks{Department of Computer Science, University of Warwick, Coventry CV4 7AL, UK. E-mail: {\tt j.w.cooper@warwick.ac.uk}.}\and
Daniel Kr\'al'\thanks{Faculty of Informatics, Masaryk University, Botanick\'a 68A, 602 00 Brno, Czech Republic, and Mathematics Institute, DIMAP and Department of Computer Science, University of Warwick, Coventry CV4 7AL, UK. E-mail: {\tt dkral@fi.muni.cz}.}\and
Ta\'isa L.~Martins\thanks{Instituto Nacional de Matem\'atica Pura e Aplicada (IMPA), Estrada Dona Castorina 110, Jardim Bot\^anico, CEP 22460-320 Rio de Janeiro, Brasil. The previous affiliation: Mathematics Institute, University of Warwick, Coventry CV4 7AL, UK. E-mail: {\tt t.lopes-martins@warwick.ac.uk}.}}
\date{}	
\maketitle

\begin{abstract}
The theory of graph limits represents large graphs by analytic objects called graphons.
Graph limits determined by finitely many graph densities, which are represented
by finitely forcible graphons, arise in various scenarios, particularly within extremal combinatorics.
Lov\'asz and Szegedy conjectured that all such graphons possess a simple structure,
e.g., the space of their typical vertices is always finite dimensional;
this was disproved by several ad hoc constructions of complex finitely forcible graphons.
We prove that any graphon is a subgraphon of a finitely forcible graphon.
This dismisses any hope for a result showing that finitely forcible graphons possess a simple structure, and
is surprising when contrasted with the fact that finitely forcible graphons form a meager set in the space of all graphons.
In addition, since any finitely forcible graphon represents the unique minimizer of some linear
combination of densities of subgraphs,
our result also shows that such minimization problems,
which conceptually are among the simplest kind within extremal graph theory,
may in fact have unique optimal solutions with arbitrarily complex structure.
\end{abstract}


\section{Introduction}

The theory of graph limits offers analytic tools to represent and analyze large graphs;
for an introduction to this area, we refer the reader to a recent monograph by Lov\'asz~\cite{bib-lovasz-book}.
Indeed, the theory also generated new tools and perspectives on many problems in mathematics and computer science.
For example, the flag algebra method of Razborov~\cite{bib-razborov07},
which bears close connections to convergent sequences of dense graphs,
catalyzed progress on many important problems in extremal combinatorics,
e.g.~\cite{bib-flag1, bib-flag2, bib-flagrecent, bib-flag3, bib-flag4, bib-flag5, bib-flag6,
bib-flag7, bib-flag8, bib-flag9, bib-flag10, bib-razborov07,bib-flag11, bib-flag12}.
In relation to computer science,
the theory of graph limits shed new light on property and parameter testing algorithms~\cite{bib-lovasz10+}.

Central to dense graph convergence is the analytic representation of the limit of a convergent sequence of dense graphs,
known as a \emph{graphon}~\cite{bib-borgs08+,bib-borgs+,bib-borgs06+,bib-lovasz06+};
we formally define this notion and other related notions in Section~\ref{sec-defs}.
We are interested in graphons that are uniquely determined (up to isomorphism) by finitely many graph densities,
which are called \emph{finitely forcible} graphons.
Such graphons are related to various problems from extremal graph theory and from graph theory in general.
For example,
for every finitely forcible graphon $W$,
there exists a linear combination of graph densities such that
the graphon $W$ is its unique minimizer.
Another result that we would like to mention in relation to finitely forcible graphons
is the characterization of quasirandom graphs in terms of graph densities by Thomason~\cite{bib-thomason, bib-thomason2},
which is essentially equivalent to stating that the constant graphon is finitely forcible by densities of $4$-vertex graphs;
also see~\cite{bib-chung89+,bib-rodl} for further results on quasirandom graphs.
Lov\'asz and S\'os~\cite{bib-lovasz08+} generalized this characterization
by showing that every step graphon, a multipartite graphon with quasirandom edge densities between its parts,
is finitely forcible.
Other examples of finitely forcible graphons are given in~\cite{bib-lovasz11+}.

Early examples of finitely forcible graphons indicated that
all finitely forcible graphons might possess a simple structure,
as formalized by Lov\'asz and Szegedy,
who conjectured the following~\cite[Conjectures~9 and 10]{bib-lovasz11+}.
\begin{conjecture}
\label{conj-compact}
The space of typical vertices of every finitely forcible graphon is compact.
\end{conjecture}
\begin{conjecture}
\label{conj-dimension}
The space of typical vertices of every finitely forcible graphon has finite dimension.
\end{conjecture}
\noindent Both conjectures were disproved through counterexample constructions~\cite{bib-inf,bib-comp}.
A stronger counterexample to Conjecture~\ref{conj-dimension} was found in~\cite{bib-reg}:
Conjecture~\ref{conj-dimension} would imply that 
the number of parts of a weak $\varepsilon$-regular partition of a finitely forcible graphon is bounded by a polynomial of $\varepsilon^{-1}$
but the construction given in~\cite{bib-reg} almost matches the best possible exponential lower bound from~\cite{bib-conlon12+}.

The purpose of this paper is to show that finitely forcible graphons can have arbitrarily complex structure.
Our main result reads as follows.
\begin{theorem}
\label{thm-main}
For every graphon $W_F$,
there exists a finitely forcible graphon $W_0$ such that $W_F$ is a subgraphon of $W_0$, and
the subgraphon is formed by a $1/14$ fraction of the vertices of $W_0$.
\end{theorem}
Theorem~\ref{thm-main} contrasts with~\cite[Theorem 7.12]{bib-lovasz11+}, which states that
the set of finitely forcible graphons is meager in the space of all graphons.
In addition,
since every finitely forcible graphon is the unique minimizer of some linear combination of densities of
subgraphs (see Proposition~\ref{prop-extr}), Theorem~\ref{thm-main} also shows that
optimal solutions of problems seeking to minimize a linear combination of densities of subgraphs,
which are among the simplest stated problems in extremal graph theory,
may have unique optimal solutions with highly complex structure.

Theorem~\ref{thm-main} also immediately implies that both conjectures presented above are false
since we can embed graphons not having the desired properties in a finitely forcible graphon.
By considering a graphon containing appropriately scaled copies of graphons
corresponding to the lower bound construction of Conlon and Fox from~\cite{bib-conlon12+},
which were described in~\cite{bib-reg}, we also obtain the following.
\begin{corollary}
\label{cor-main}
For every non-decreasing function $f:\RR\to\RR$ tending to infinity,
there exist a finitely forcible graphon $W$ and positive reals $\varepsilon_i$ tending to $0$
such that every weak $\varepsilon_i$-regular partition of $W$
has at least $2^{\Omega\left(\frac{\varepsilon_i^{-2}}{f(\varepsilon_i^{-1})}\right)}$ parts.
\end{corollary}
Since every fixed graphon has weak $\varepsilon$-regular partitions with $2^{o(\varepsilon^{-2})}$ parts,
Corollary~\ref{cor-main} gives the best possible dependance on $\varepsilon^{-1}$.

The proof of Theorem~\ref{thm-main} builds on the methods introduced in~\cite{bib-comp}, which were further developed and formalized in~\cite{bib-inf}.
In particular, the proof uses the technique of decorated constraints, which we present in Subsection~\ref{sec-constraints}. 
The main idea of the proof is the following. The graphon $W_F$ is determined up to a set of measure zero by its density
in squares with coordinates being the inverse powers of two. 
The countable sequence of such densities can be encoded into a single real number between $0$ and $1$,
which will be embedded as the density of a suitable part of the graphon $W_0$.
We then set up the structure of $W_0$ in a way that
this encoding restricts the densities inside another part of $W_0$ rendering $W_F$ unique up to a set of measure zero.
While this approach seems uncomplicated upon first glance, the proof hides a variety of additional ideas and technical details.
The reward is a result enabling the embedding of any graphon in a finitely forcible graphon with no additional effort.

\section{Preliminaries}
\label{sec-defs}

We devote this section to introducing notation used throughout the paper.
Let us start with some basic general notation.
The set of integers from $1$ to $k$ will be denoted by $[k]$,
the set of all positive integers by $\NN$ and the set of all non-negative integers by $\NN_0$.
All measures considered in this paper are the Borel measures on $\RR^d$, $d\in\NN$.
If a set $X\subseteq\RR^d$ is measurable, then $|X|$ denotes its measure, and
if $X$ and $Y$ are two measurable sets, then we write $X\sqsubseteq Y$ if $|X\setminus Y|=0$.

The \emph{order} of a graph $G=(V,E)$, denoted by $|G|$, is the number of its vertices, and
its \emph{size}, denoted by $||G||$, is the number of edges. 
Given two graphs $H$ and $G$,
the \emph{density} of $H$ in $G$
is the probability that a uniformly chosen $|H|$-tuple of vertices of $G$ induces a subgraph isomorphic to $H$;
the density of $H$ in $G$ is denoted by $d(H,G)$.
We adopt the convention that if $|H|>|G|$, then $d(H,G)=0$.

A sequence of graphs $(G_n)_{n\in \NN}$
is \emph{convergent} if the sequence $ d(H,G_n)$ converges for every graph $H$.
We will require that the orders of graphs in a convergent sequence tend to infinity.
A convergent sequence of graphs can be associated with an analytic limit object, which is called a graphon.
A \emph{graphon} is a symmetric measurable function $W$ from the unit square $[0,1]^2$ to the unit interval $[0,1]$,
where \emph{symmetric} refers to the property that $W(x,y) = W(y,x)$ for all $x, y \in [0, 1]$.
In what follows, we will often refer to the points of $[0,1]$ as \emph{vertices}.
One may view the values of $W(x,y)$ as the density between different parts of a large graph represented by $W$.
To aid the transparency of our ideas, we often include a visual representation of graphons that we consider:
the domain of a graphon $W$ is represented as a unit square $[0,1]^2$ with the origin $(0,0)$ in the top left corner, and
the values of $W$ are represented by an appropriate shade of gray (ranging from white to black),
with $0$ represented by white and $1$ by black.

Given a graphon $W$, a \emph{$W$-random graph} of order $n$ is a graph obtained from $W$
by sampling $n$ vertices $v_1, v_2,\ldots, v_n \in [0, 1]$ independently and uniformly at random and
joining vertices $v_i$ and $v_j$ by an edge with probability $W(v_i, v_j)$ for all $i,j\in [n]$.
The density of a graph $H$ in a graphon $W$, denoted by $d(H, W)$,
is the probability that a $W$-random graph of order $|H|$ is isomorphic to $H$.
Note that the expected density of $H$ in a $W$-random graph of order $n\ge |H|$ is equal to $d(H,W)$.
We say that a convergent sequence $(G_n)_{n\in \NN}$ converges to a graphon $W$
if $$\lim_{n\to\infty} d(H,G_n)=d(H,W)$$
for every graph $H$.
It is not hard to show that if $W$ is a graphon,
then the sequence of $W$-random graphs with increasing orders
is convergent with probability one and the graphon $W$ is its limit.

We now present graphon analogues of several graph theoretic notions.
The \emph{degree} of a vertex $x \in [0, 1]$ is defined as
$$\deg_W(x) = \int_{[0,1]} W(x,y) \dd y.$$
Note that the degree is well-defined for almost all vertices of $W$ and
if $x$ is chosen to be a vertex of an $n$-vertex $W$-random graph, then its expected degree is $(n-1)\cdot\deg_W(x)$.
When it is clear from the context which graphon we are referring to, we will omit the subscript,
i.e., we just write $\deg(x)$ instead of $\deg_W(x)$.
We define the \emph{neighborhood} $N_W(x)$ of a vertex $x\in [0,1]$ in a graphon $W$
as the set of vertices $y\in [0,1]$ such that $W(x,y)>0$.
In our considerations,
we will often analyze a restriction of a graphon to the substructure induced by a pair of measurable subsets $A$ and $B$ of $[0,1]$.
If $W$ is a graphon and $A$ is a non-null measurable subset of $[0,1]$,
then the \emph{relative degree} of a vertex $x\in [0,1]$ with respect to $A$ is
$$\deg_W^A(x) = \frac{\int_A W(x,y) \dd y}{|A|},$$
i.e., the measure of the neighbors of $x$ in $A$ normalized by the measure of $A$.
Similarly, $N_W^A(x)=N_W(x)\cap A$ is the \emph{relative neighborhood} of $x$ with respect to $A$.
Note that $\deg_W^A(x)\cdot |A|\le |N_W^A(x)|$ and the inequality can be strict.
Again, we drop the subscripts when $W$ is clear from the context.

Two graphons $W_1$ and $W_2$ are \emph{weakly isomorphic} if $d(H, W_1) = d(H, W_2)$ for every graph $H$.
Borgs, Chayes and Lov\'asz~\cite{bib-borgs10+} have shown that two graphons $W_1$ and $W_2$ are weakly isomorphic
if and only if there exist measure preserving maps $\varphi_1, \varphi_2: [0, 1]\to [0, 1]$ such that
$W_1(\varphi_1(x),\varphi_1(y))=W_2(\varphi_2(x),\varphi_2(y))$ for almost every $(x, y) \in [0, 1]^2$.
Graphons that can be uniquely determined up to a weak isomorphism by fixing the densities of a finite set of graphs
are called \emph{finitely forcible graphons} and are the central object of this paper.
Observe that a graphon $W$ is finitely forcible if and only if
there exist graphs $H_1,\ldots,H_k$ such that
if a graphon $W'$ satisfies $d(H_i, W')=d(H_i,W)$ for $i\in [k]$,
then $d(H, W')=d(H,W)$ for every graph $H$.
A less obvious characterization of finitely forcible graphons is the following.
\begin{proposition}
\label{prop-extr}
A graphon $W$ is finitely forcible if and only if
there exist graphs $H_1,\ldots,H_k$ and reals $\alpha_1,\ldots,\alpha_k$ such that
$$\sum_{i=1}^k\alpha_i d(H_i,W)\le\sum_{i=1}^k\alpha_i d(H_i,W')$$
for every graphon $W'$ and the equality holds only if $W$ and $W'$ are weakly isomorphic.
\end{proposition}

\subsection{Finite forcibility and decorated constraints}
\label{sec-constraints}

Decorated constraints have been introduced and developed in~\cite{bib-inf,bib-comp} as
a method of showing finite forcibility of graphons.
This method uses the language of the flag algebra method of Razborov,
which, as we have mentioned earlier, has had many substantial applications in extremal combinatorics.
We now present the notion of decorated constraints, partially following the lines of~\cite{bib-inf} in our exposition.

A \emph{density expression} is iteratively defined as follows:
a real number or a graph are density expressions, and
if $D_1$ and $D_2$ are density expressions, then so are $D_1 + D_2$ and $D_1 \cdot D_2$.
The value of a density expression with respect to a graphon $W$
is the value obtained by substituting for each graph its density in the graphon $W$.
A \emph{constraint} is an equality between two density expressions.
A graphon $W$ \emph{satisfies} a constraint if the density expressions on the two sides of the constraints have the same value.
If $\CC$ is a finite set of constraints such that 
there exists a unique (up to weak isomorphism) graphon $W$ that satisfies all constraints in $\CC$,
then the graphon $W$ is finitely forcible~\cite{bib-comp};
in particular, $W$ can be forced by specifying the densities of graphs appearing in the constraints in $\CC$.

A graphon $W$ is said to be \emph{partitioned}
if there exist $k \in \NN$, positive reals $a_1,\ldots, a_k$ with $a_1+\cdots+a_k=1$, and
distinct reals $d_1,\ldots,d_k\in [0, 1]$, such that the set of vertices in $W$ with degree $d_i$ has measure $a_i$.
The set of all vertices with degree $d_i$ will be referred to as a \emph{part};
the \emph{size} of a part is its measure and its \emph{degree} is the common degree of its vertices.
The following lemma was proved in~\cite{bib-inf,bib-comp}.

\begin{lemma}
\label{lm-partition}
Let $a_1,\ldots,a_k$ be positive real numbers summing to one and let $d_1,\ldots,d_k\in [0,1]$ be distinct reals.
There exists a finite set of constraints $\CC$ such that 
any graphon satisfying all constraints in $\CC$ is a partitioned graphon with parts of sizes $a_1,\ldots,a_k$ and degrees $d_1,\ldots,d_k$, and
every partitioned graphon with parts of sizes $a_1,\ldots,a_k$ and degrees $d_1,\ldots,d_k$ satisfies all constraints in $\CC$.
\end{lemma}

We next introduce a formally stronger version of constraints, called decorated constraints.
Fix $a_1,\ldots,a_k$ and $d_1,\ldots,d_k$ as in Lemma~\ref{lm-partition}.
A \emph{decorated} graph is a graph $G$ with $m\leq|G|$ distinguished vertices labeled from 1 to $m$, which are called \emph{roots}, and
with each vertex assigned one of the $k$ parts, which is referred to as the decoration of a vertex.
Note that the number $m$ can be zero in the definition of a decorated graph, i.e., a decorated graph can have no roots.
Two decorated graphs are \emph{compatible} if the subgraphs induced by their roots are isomorphic
through an isomorphism preserving the labels (the order of the roots) and the decorations (the assignment of parts).
A \emph{decorated constraint} is an equality between two density expressions that contain decorated graphs instead of ordinary graphs and
all the decorated graphs appearing in the constraint are compatible.

Consider a partitioned graphon $W$ with parts of sizes $a_1,\ldots,a_k$ and degrees $d_1,\ldots,d_k$, and
a decorated constraint $C$.
Let $H_0$ be the (decorated) graph induced by the roots of the decorated graphs in the constraint, and
let $v_1,\ldots,v_m$ be the roots of $H_0$.
We say that the graphon $W$ \emph{satisfies} the constraint $C$
if the following holds for almost every $m$-tuple $x_1,\ldots,x_m\in [0,1]$ such that
$x_i$ belongs to the part that $v_i$ is decorated with, $W(x_i,x_j)>0$ for every edge $v_iv_j$ and
$W(x_i,x_j)<1$ for every non-edge $v_iv_j$:
if each decorated graph $H$ in $C$ is replaced with the probability that a $W$-random graph is the graph $H$
conditioned on the event that the roots are chosen as the vertices $x_1,\ldots,x_m$ and they induce the graph $H_0$, and that
each non-root vertex is randomly chosen from the part of $W$ that is decorated with,
then the left and right hand sides of the constraint $C$ have the same value.

We now give an example of evaluating a decorated constraint. 
Consider a partitioned graphon $W$, which is depicted in Figure~\ref{fig:exConst}, with parts $A$ and $B$ each of size $1/2$;
the graphon $W$ is equal to $1/2$ on $A^2$, to $1/3$ on $A\times B$, and to $1$ on $B^2$.
Let $H$ be the decorated graph with two adjacent roots both decorated with $A$ and two adjacent non-root vertices $v_1$ and $v_2$ both decorated with $B$
such that $v_1$ is adjacent to only one of the roots and $v_2$ is adjacent to both roots;
the decorated graph $H$ is also depicted in Figure~\ref{fig:exConst}.
If $H$ appears in a decorated constraint,
then its value is independent of the choice of the roots in the part $A$ and is always equal to $2/81$,
which is the probability as defined in the previous paragraph.

\begin{figure}[htbp]
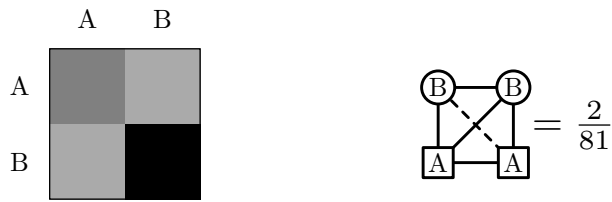

	\begin{center}
		\hskip -50mm \epsfbox{ffuniversal-48.mps} 
		\vskip -18mm \hskip 60mm \epsfbox{ffuniversal-150.mps}		
	\end{center}
	\caption{An example of evaluating a decorated constraint.
	         The root vertices are depicted by squares and the non-root vertices by circles.
		 The graphon is equal to $1/2$ on $A^2$, to $1/3$ on $A\times B$, and to $1$ on $B^2$.\label{fig:exConst}}
\end{figure}

Note that the condition on the $m$-tuple $x_1,\ldots,x_m$ is equivalent to that
there is a positive probability that a $W$-random graph with the vertices $x_1,\ldots,x_m$ is $H_0$.
Also note that, unlike in the definition of the density of a graph in a graphon,
we do not allow permuting any vertices.
For example, if $W$ is the graphon (with a single part) that is equal to $p\in [0,1]$ almost everywhere,
then the cherry $K_{1,2}$ with each vertex decorated with the single part of $W$
would take the value $p^2(1-p)$ in a decorated constraint.

The next lemma, proven in~\cite{bib-comp}, asserts that
every decorated constraint is equivalent to a non-decorated constraint.
\begin{lemma}
\label{lm-decorated}
Let $k\in \NN$, let $a_1,\ldots,a_k$ be positive real numbers summing to one, and
let $d_1,\ldots,d_k$ be distinct reals between zero and one.
For every decorated constraint $C$, there exists a constraint $C'$ such that
any partitioned graphon $W$ with parts of sizes $a_1,\ldots,a_k$ and degrees $d_1,\ldots,d_k$
satisfies $C$ if and only if it satisfies $C'$.
\end{lemma}
In particular, if a graphon $W$ is a unique partitioned graphon up to weak isomorphism that satisfies a finite collection of decorated constraints,
then it is a unique graphon satisfying a finite collection of ordinary constraints by Lemmas~\ref{lm-partition} and~\ref{lm-decorated}, and
hence $W$ is finitely forcible.

We will visualize decorated constraints using the convention from~\cite{bib-reg}, which we now describe and 
have already used in Figure~\ref{fig:exConst}.
The root vertices of decorated graphs in a decorated constraint will be depicted by squares and the non-root vertices by circles;
each vertex will be labeled with its decoration, i.e., the part that it should be contained in.
The roots will be in all the decorated graphs in the constraint in the same mutual position,
so it is easy to see the correspondence of the roots of different decorated graphs in the same constraint.
A solid line between two vertices represents an edge, and a dashed line represents a non-edge.
The absence of a line between two root vertices indicates that the decorated constraint
should hold for both the root graph containing this edge and not containing it.
Finally, the absence of a line between a non-root vertex and another vertex represents the sum of decorated graphs
with this edge present and without this edge.
If there are $k$ such lines absent,
the figure represents the sum of $2^k$ possible decorated graphs with these edges present or absent.

We finish this subsection with two auxiliary lemmas.
The first is a lemma stated in~\cite{bib-reg},
which essentially states that if a graphon $W_0$ is finitely forcible in its own right,
then it may be forced on a part of a partitioned graphon without altering the structure of the rest of the graphon.
\begin{lemma}
\label{lm-subforcing}
Let $k\in \NN$, $m\in [k]$, let $a_1,\ldots,a_k$ be positive real numbers summing to one, and
let $d_1,\ldots,d_k$ be distinct reals between zero and one.
If $W_0$ is a finitely forcible graphon,
then there exists a finite set $\CC$ of decorated constraints such that
any partitioned graphon $W$ with parts of sizes $a_1,\ldots,a_k$ and degrees $d_1,\ldots,d_k$ satisfies $C$ if and only if
there exist measure preserving maps $\varphi_0:[0,1]\to [0,1]$ and $\varphi_m:[0,a_m]\to A_m$ such that
$W(\varphi_m(xa_m),\varphi_m(ya_m))=W_0(\varphi_0(x),\varphi_0(y))$ for almost every $(x,y)\in [0,1]^2$,
where $A_m$ is the $m$-th part of $W$.
\end{lemma}
\noindent Note that Lemma~\ref{lm-decorated} implies that
the set $\CC$ of decorated constraints from Lemma~\ref{lm-subforcing}
can be turned into a set of ordinary (i.e., non-decorated) constraints.

The second lemma is implicit in~\cite[proof of Lemma 3.3]{bib-lovasz11+};
its special case has been stated explicitly in, e.g.,~\cite[Lemma 8]{bib-reg}.
\begin{lemma}
\label{lm-typ-pairs}
Let $X,Z\subseteq\RR$ be two measurable non-null sets, and
let $F:X\times Z\to [0,1]$ be a measurable function.
If there exists $C\in\RR$ such that
$$\int_Z F(x,z)F(x',z) \dd z = C$$
for almost every $(x,x')\in X^2$, then
$$\int_Z F(x,z)^2 \dd z = C$$
for almost every $x\in X$.
\end{lemma}

\subsection{Regularity partitions and step functions}
\label{sec-regularity}

A {\em step function} $W:[0,1]^2\to [-1,1]$ is a measurable function such that
there exists a partition of $[0,1]$ into measurable non-null sets $U_1,\ldots,U_k$ that
$W$ is constant on $U_i\times U_j$ for every $i,j\in [k]$.
A non-negative symmetric step function is a {\em step graphon}.
If $W$ is a step function (in particular, $W$ can be a step graphon) and $A$ and $B$ two measurable subsets of $[0,1]$,
then the {\em density $d_W(A,B)$ between $A$ and $B$} is defined to be
$$d_W(A,B)=\int\limits_{A\times B} W(x,y)\dd x\dd y\;\mbox{.}$$
We will omit $W$ in the subscript if $W$ is clear from the context.
Note that it always holds that $|d(A,B)|\le |A|\cdot |B|$.
We remark that the definition of $d_W(A,B)$ and the definition of the cut norm given below
extend to a wider class of functions from $[0,1]^2$ to $\RR$, which are called kernels;
since we will use these definition for step functions, we present them in the setting of step functions only.
A step function $W'$ {\em refines} a step function $W$ with parts $U_1,\ldots,U_k$,
if each part of $W'$ is a subset of one of the parts of $W$ and
the density $d_W(U_i,U_j)$ between $U_i$ and $U_j$ is equal to
the weighted average of the densities between the pairs of those parts of $W'$ that are subsets of $U_i$ and $U_j$, respectively.

We next recall the notion of the cut norm.
If $W$ is a step function, then the {\em cut norm} of $W$, denoted by $\cut{W}$, is
$$\sup_{A,B\subseteq [0,1]}\left|\int_{A\times B}W(x,y)\dd x\dd y\right|\;\mbox{,}$$
where the supremum is taken over all measurable subsets $A$ and $B$ of $[0,1]$.
The supremum in the definition is always attained and
the cut norm induces the same topology on the space of step functions as the $L_1$-norm;
this can be verified following the lines of the analogous arguments for graphons in~\cite[Chapter 8]{bib-lovasz-book}.
We emphasize that
we do not allow applying a measure preserving transformation to the domain of graphons
unlike in the definition of the cut distance.
It can be shown~\cite[Lemma 10.23]{bib-lovasz-book} that
if $H$ is a $k$-vertex graph and $W$ and $W'$ are two graphons,
then
$$\left|d(H,W)-d(H,W')\right|\le {k\choose 2}\cut{W-W'}\;\mbox{.}$$
Finally,
we will say that two graphons $W$ and $W'$ are {\em $\varepsilon$-close} if $\cut{W-W'}\le\varepsilon$.

A partition of $[0,1]$ into measurable non-null sets $U_1,\ldots,U_k$ is said to be {\em $\varepsilon$-regular}
if
$$\left|d(A,B)-\sum_{i,j\in [k]}\frac{d(U_i,U_j)}{|U_i||U_j|}|U_i\cap A||U_j\cap B|\right|\le\varepsilon$$
for every two measurable subsets $A$ and $B$ of $[0,1]$.
In other words, the step graphon $W'$ with parts $U_1,\ldots,U_k$ that is equal to $\frac{d(U_i,U_j)}{|U_i||U_j|}$ on $U_i\times U_j$
is $\varepsilon$-close to $W$ in the cut norm metric.
In particular, the step graphon $W'$ determines the densities of $k$-vertex graphs in $W$
up to an additive error of ${k\choose 2}\varepsilon$.

The Weak Regularity Lemma of Frieze and Kannan~\cite{bib-frieze99+} extends to graphons
as follows (see~\cite[Section 9.2]{bib-lovasz-book} for further details):
for every $\varepsilon>0$,
there exists $K\le 2^{O\left(\varepsilon^{-2}\right)}$, which depends on $\varepsilon$ only, such that
every graphon has an $\varepsilon$-regular partition with at most $K$ parts.
This dependence of $K$ on $\varepsilon$ is best possible up to a constant factor in the exponent~\cite{bib-conlon12+}.
We will need a slightly stronger version of this statement, which we formulate as a proposition;
its proof is an easy modification of a proof of the standard version of the statement,
e.g., the one presented in~\cite[Section 9.2]{bib-lovasz-book}.
\begin{proposition}
\label{prop-regularity}
For every $\varepsilon>0$ and $k\in\NN$,
there exists $K\in\NN$ such that for every graphon $W$ and every partition $U_1,\ldots,U_k$ of $[0,1]$ into disjoint measurable non-null sets,
there exist an $\varepsilon$-regular partition $U'_1,\ldots,U'_{K'}$ of $[0,1]$ with $K'\le K$ such that
every part $U'_{i}$, $i\in [K']$, is a subset of one of the parts $U_1,\ldots,U_k$.
\end{proposition}

For a step function $W$, we define $d(\Gamma_4,W)$ to be the following integral:
$$d(\Gamma_4,W)=\int_{[0,1]^4}W(x,y)W(x',y)W(x,y')W(x',y')\dd x\dd x'\dd y\dd y'\;\mbox{.}$$
Note that the definition of $d(\Gamma_4,W)$ coincides with the standard definition of $t(C_4,W)$,
see e.g.~\cite{bib-lovasz-book}.
In particular, if $W$ is a graphon, then it holds that
$$d(\Gamma_4,W)=\frac{1}{3}d(C_4,W)+\frac{1}{3}d(K_4^-,W)+d(K_4,W)\;\mbox{,}$$
where $K_4^-$ is the graph obtained from $K_4$ by removing one of its edges.
In particular, $d(\Gamma_4,W)$ can be understood as the density of non-induced $C_4$ in the graphon $W$,
since it is equal to the expected density of non-induced copies of $C_4$ in a $W$-random graph.
If $W$ is a step function, then $d(\Gamma_4,W)\le 4 \cut{ W}$.
However, the converse also holds: $d(\Gamma_4,W)\ge\cut{W}^4$;
we refer e.g.~to~\cite[Section 8.2]{bib-lovasz-book},
where a proof for symmetric step functions $W$ is given and this proof readily extends to the general case.
Lemma~\ref{lm-C4}, which we present further, aims at a generalization of this statement to step graphons.
Before we can state this lemma, we need to prove two auxiliary lemmas,
which we state for matrices rather than step functions for simplicity.

\begin{lemma}
\label{lm-C4-aux}
Let $M$ be a $K\times K$ real matrix and let $i,j\in [K]$.
Define $N$ to be the following $K\times K$ matrix:
$$N_{x,y}=\left\{
\begin{array}{cl}
\frac{M_{i,y}+M_{j,y}}{2} & \mbox{if $x=i$ or $x=j$, and}\\
M_{x,y} & \mbox{otherwise.}
\end{array}
\right.$$
It holds that $\Tr MM^TMM^T\ge \Tr NN^TNN^T$.
\end{lemma}

\begin{proof}
Set $M(x,y)$, $x,y\in [K]$, to be the following quantity:
$$M(x,y)=\sum_{z=1}^K M_{x,z}M_{y,z}\;\mbox{,}$$
and define $N(x,y)$, $x,y\in [K]$, in the analogous way.
Observe that
$$\Tr MM^TMM^T-\Tr NN^TNN^T=\sum_{x,y=1}^KM(x,y)^2-N(x,y)^2\;\mbox{.}$$
We now analyze the difference on the right hand side of the equality
by grouping the terms on the right hand side into disjoint sets such that
the sum of the terms in each set is non-negative.

The terms with $x,y\in [K]\setminus\{i,j\}$ form singleton sets;
note that $M(x,y)=N(x,y)$ for each such term.
Fix $x\in [K]\setminus\{i,j\}$ and consider the two terms corresponding to $y=i$ and $y=j$.
It follows that
$$M(x,i)^2+M(x,j)^2-N(x,i)^2-N(x,j)^2 = $$
$$M(x,i)^2+M(x,j)^2-2 \left(\sum_{z=1}^K M_{xz}\frac{M_{i,z}+M_{j,z}}{2} \right)^2 = $$
$$M(x,i)^2+M(x,j)^2-\frac{1}{2}\left(M(x,i)+M(x,j)\right)^2 = $$
$$\frac{1}{2}M(x,i)^2+\frac{1}{2}M(x,j)^2-M(x,i)M(x,j)=\frac{1}{2}\left(M(x,i)-M(x,j)\right)^2 \;\mbox{.}$$
Hence, the sum of any pair of such terms is non-negative.
The analysis of the terms with $y\in [K]\setminus\{i,j\}$ and $x=i$ or $x=j$ is symmetric.

The remaining four terms that have not been analyzed are the terms corresponding to the following pairs $(x,y)$:
$(i,i)$, $(i,j)$, $(j,i)$ and $(j,j)$.
In this case, we obtain the following:
$$M(i,i)^2+2M(i,j)^2+M(j,j)^2-N(i,i)^2-2N(i,j)^2-N(j,j)^2=$$
$$M(i,i)^2+2M(i,j)^2+M(j,j)^2-4\left(\frac{M(i,i)+2M(i,j)+M(j,j)}{4}\right)^2=$$
$$\frac{1}{4}\left(M(i,i)-M(j,j)\right)^2+
  \frac{1}{2}\left(M(i,i)-M(i,j)\right)^2+
  \frac{1}{2}\left(M(j,j)-M(i,j)\right)^2\;\mbox{.}$$
Hence, the sum of these four terms is also non-negative, and the lemma follows.
\end{proof}

The next lemma follows by repeatedly applying Lemma~\ref{lm-C4-aux}
to pairs of rows of the matrix $M$ with indices from the same set $A_i$ and
to pairs of rows of the matrix $M^T$ with indices from the same set $B_i$, and
considering the limit matrix $N$.

\begin{lemma}
\label{lm-C4-aux-iter}
Let $M$ be a $K\times K$ real matrix.
Further, let $X_1,\ldots,X_k$ be a partition of $[K]$ into $k$ disjoint sets and
let $Y_1,\ldots,Y_{\ell}$ be a partition of $[K]$ into $\ell$ disjoint sets.
Define the $K\times K$ matrix $N$ as follows. If $x\in X_i$, $y\in Y_j$,
then
$$N_{x,y}=\frac{1}{|X_i|\cdot |Y_j|} \sum_{x'\in X_i,y' \in Y_j} M_{x',y'}\;\mbox{.}$$
It holds that $\Tr MM^TMM^T=\Tr M^TMM^TM\ge \Tr NN^TNN^T=\Tr N^TNN^TN$.
\end{lemma}

The following auxiliary lemma can be viewed as an extension of \cite[Lemma 8.12]{bib-lovasz-book},
which states that $d(\Gamma_4,W)\ge\cut{W}^4$ for every graphon $W$,
from the zero graphon to general step graphons (consider the statement for $W_0$ being the zero graphon).
We remark that we have not tried to obtain the best possible dependence on the parameter $\varepsilon$ in the statement of the lemma.
The lemma also holds in a more general setting, where the parts of graphons are not required to be of the same size.

\begin{lemma}
\label{lm-C4}
Let $W_0$ be a step graphon with all parts of the same size, and
$W$ a step graphon refining $W_0$ such that all parts of $W$ have the same size.
If $\cut{W-W_0}\ge\varepsilon$, then $d(\Gamma_4,W)\ge d(\Gamma_4,W_0)+\varepsilon^4/8$.
\end{lemma}

\begin{proof}
Since $\cut{W-W_0}\ge\varepsilon$, there exist two measurable subsets $A$ and $B$ of $[0,1]$ such that
\begin{equation}
\left|\int_{A\times B}W(x,y)-W_0(x,y)\dd x\dd y\right|\ge\varepsilon\;\mbox{.}\label{eq-C4-norm}
\end{equation}
Let $U$ be one of the parts of the graphon $W$.
Depending whether $\int_{U\times B}W-W_0\dd x\dd y$ is positive or negative,
replacing $A$ with either $A\cup U$ or $A\setminus U$ does not decrease the integral in (\ref{eq-C4-norm}).
Hence, we can assume that each part of $W$ is either a subset of $A$ or is disjoint from $A$, and
the same holds with respect to $B$ (but different parts $U$ of $W$ may be contained in $A$ and $B$).

Let $k$ be the number of parts of $W_0$ and $K$ the number of parts $W$.
Further, let $M$ be the $K\times K$ matrix such that
the entry $M_{i,j}$, $i,j\in K$, is the density of $W$ between its $i$-th and the $j$-th parts, and
let $P$ be the $K\times K$ matrix such that
$P_{i,j}$, $i,j\in K$, is the density of $W_0$ between the $i$-th and the $j$-th parts of $W$.
Let $U_i$, $i\in [k]$, be the subset of $[K]$ containing the indices of the parts of $W$ contained in the $i$-th part of $W_0$.
Observe that both matrices $M$ and $P$ are symmetric and
the matrix $P$ is constant on each submatrix indexed by pairs from $U_i\times U_j$ for some $i,j\in [k]$.
Since $d(\Gamma_4,W)=\Tr M^4$ and $d(\Gamma_4,W_0)=\Tr P^4$,
our goal is to show that $\Tr M^4-\Tr P^4\ge \varepsilon^4/8$.
Finally,
let $A'$ be the indices of parts of $W$ contained in $A$, and
let $B'$ be the indices of parts of $W$ contained in $B$.
Observe that (\ref{eq-C4-norm}) yields that the sum of the entries of the matrix $M-P$ 
with the indices in $A'\times B'$ is either at least $\varepsilon$ or at most $-\varepsilon$.

Let $N$ be the matrix from the statement of Lemma~\ref{lm-C4-aux-iter} for the matrix $M$,
$X_i=\{i\}$, $i\in [K]$, and $Y_{j}=U_j$, $j\in [k]$.
Let $\varepsilon_1$ be the sum of the entries of the matrix $M-N$
with the indices in $A'\times B'$, and
let $\varepsilon_2$ be the sum of the entries of the matrix $N-P$
with the indices in $A'\times B'$.
Note that $|\varepsilon_1+\varepsilon_2|\ge\varepsilon$,
which implies that $|\varepsilon_1|+|\varepsilon_2|$ is at least $\varepsilon$.
By Lemma~\ref{lm-C4-aux-iter}, it holds that $\Tr M^4-\Tr NN^TNN^T\ge 0$.
Since $P^T$ can be obtained from the matrix $N^T$ by applying Lemma~\ref{lm-C4-aux-iter}
with $X_i=\{i\}$, $i\in [K]$, and $Y_j=U_j$, $j\in [k]$, it follows that
$\Tr N^TNN^TN-\Tr P^4=\Tr NN^TNN^T-\Tr P^4\ge 0$.

We now show that $\Tr M^4-\Tr NN^TNN^T\ge\varepsilon_1^4$.
Let $Q=M-N$.
We now want to analyze the entries of the matrix $(N+\alpha Q)(N+\alpha Q)^T$ for $\alpha\in [0,1]$.
Fix $x,y\in [K]$ and observe that the entry in the $x$-th row and the $y$-th column of
the matrix $(N+\alpha Q)(N+\alpha Q)^T$ is equal to
$$\sum_{j=1}^{k}\sum_{z\in U_j}(N+\alpha Q)_{x,z}(N+\alpha Q)_{y,z}\;\mbox{.}$$
The definition of the matrix $N$ implies that
$$\sum_{z\in U_j}Q_{x,z}=\sum_{z\in U_j}Q_{y,z}=0$$
for every $j\in [k]$.
It also holds that $N_{x,z}=N_{x,z'}$ and $N_{y,z}=N_{y,z'}$ for any $z$ and $z'$ from the same set $U_j$, $j\in [k]$,
which implies that the entry of the matrix $(N+\alpha Q)(N+\alpha Q)^T$ in the $x$-th row and the $y$-th column is
$$\sum_{z=1}^{K}N_{x,z}N_{y,z}+\alpha^2Q_{x,z}Q_{y,z}\;\mbox{.}$$
Hence, we conclude that $(N+\alpha Q)(N+\alpha Q)^T=NN^T+\alpha^2QQ^T$.
It follows that 
\begin{eqnarray}
& \Tr (N+\alpha Q)(N+\alpha Q)^T(N+\alpha Q)(N+\alpha Q)^T = & \nonumber\\
& \Tr NN^TNN^T+2\alpha^2\Tr NN^TQQ^T+\alpha^4\Tr QQ^TQQ^T\;\mbox{.} & \label{eq-poly}
\end{eqnarray}
By Lemma~\ref{lm-C4-aux-iter} applied with $M=N+\alpha Q$ and the same sets $X_i$ and $Y_j$ as earlier,
$$\Tr (N+\alpha Q)(N+\alpha Q)^T(N+\alpha Q)(N+\alpha Q)^T-\Tr NN^TNN^T\ge 0$$
for every $\alpha\ge 0$, which implies that $\Tr NN^TQQ^T\ge 0$.
In particular, we obtain from (\ref{eq-poly}) for $\alpha=1$ that
\begin{eqnarray}
& \Tr M^4-\Tr NN^TNN^T= & \nonumber \\
& \Tr (N+\alpha Q)(N+\alpha Q)^T(N+\alpha Q)(N+\alpha Q)^T-\Tr NN^TNN^T\ge & \nonumber\\
& \Tr QQ^TQQ^T\;\mbox{.}\label{eq-almost1}
\end{eqnarray}
Since the cut-norm of the step graphon corresponding to $Q$ is at least $\varepsilon_1$,
it follows that $\Tr QQ^TQQ^T\ge\varepsilon_1^4$.

Applying the symmetric argument to the matrices $P^T$ and $N^T=N$, we obtain that
\begin{equation}
\Tr NN^TNN^T-\Tr P^4\ge \Tr (N-P)(N-P)^T(N-P)(N-P)^T\ge \varepsilon_2^4\;\mbox{.}\label{eq-almost2}
\end{equation}
Since $\Tr M^4-\Tr NN^TNN^T\ge 0$ and $\Tr N^TNN^TN-\Tr P^4\ge 0$,
we obtain from (\ref{eq-almost1}) and (\ref{eq-almost2}) using $|\varepsilon_1|+|\varepsilon_2|\ge\varepsilon$ that
$\Tr M^4-\Tr P^4\ge\varepsilon_1^4+\varepsilon_2^4\ge\varepsilon^4/8$, as desired.
\end{proof}

\section{General setting of the proof of Theorem~\ref{thm-main}}
\label{sec-initial}

In this section,
we provide a general overview of the structure of the graphon $W_0$ from Theorem~\ref{thm-main} and
the proof of Theorem~\ref{thm-main}.
The visualization of the structure of the graphon $W_0$ can be found in Figure~\ref{fig:graphon_univ}.
The proof of Theorem~\ref{thm-main} is spread through Sections~\ref{sec-initial}--\ref{sec-white},
with this section containing its initial steps.

\begin{figure}[htbp]
\begin{center}
\epsfbox{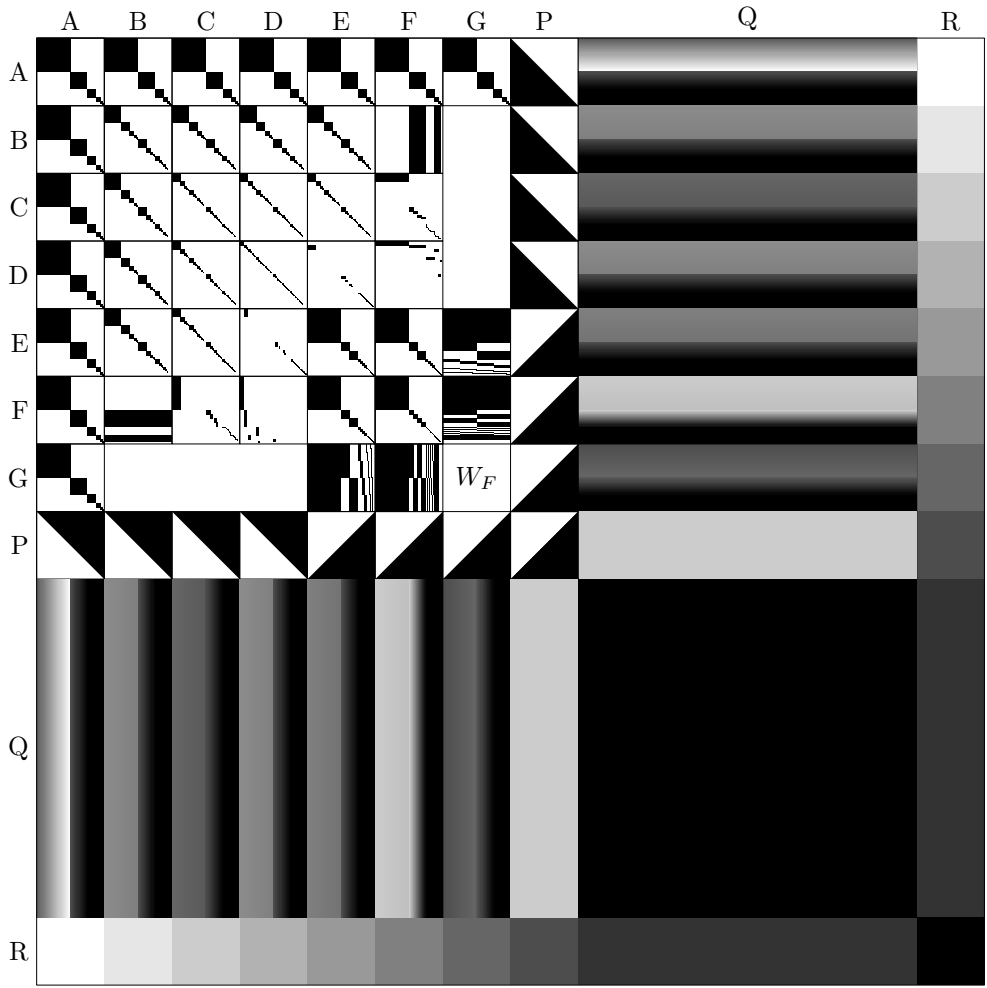}\vskip 10mm
\end{center}
\caption{The sketch of the graphon $W_0$ from Theorem~\ref{thm-main}. \label{fig:graphon_univ}}
\end{figure}

Fix a graphon $W_F$.
The graphon $W_0$ is a partitioned graphon with $10$ parts denoted by capital letters from $A$ to $R$.
Each part except for $Q$ has size $1/14$, and the size of $Q$ is $5/14$.
If $X,Y\in\{A,\ldots,G,P,Q,R\}$ are two parts,
the restriction of the graphon $W_0$ to $X\times Y$ will be referred to as the \emph{tile $X\times Y$}.
The graphon $W_F$ will be contained in the tile $G\times G$ of the graphon $W_0$.
The degrees of the parts (i.e., the degrees of the vertices forming the parts) are given in Table~\ref{tab-degrees};
the degree of $Q$ is at least $5/14+8/252$, i.e., larger than the degree of any other part, and will be fixed later in the proof.

\begin{table}[htbp]
\begin{center}
\begin{tabular}{|l|c|c|c|c|c|c|c|c|c|c|}
\hline
Part & $A$ &$B$ &$C$ & $D$ & $E$ & $F$ & $G$ & $P$ &$Q$ & $R$ \\\hline & & & & & & & & 
&\\[-1em]
Degree & $\frac{90}{252}$ & $\frac{91}{252}$ & $\frac{92}{252}$ & $\frac{93}{252}$ & $\frac{94}{252}$ & 
$\frac{95}{252}$ & $\frac{96}{252}$  & $ \frac{97}{252}$ & $\ge$ $\frac{98}{252}$& 
$\frac{77}{252}$\\[.15em]\hline
\end{tabular}
\end{center}
\caption{The degrees of the vertices in the parts of the graphon $W_0$ from the proof of Theorem~\ref{thm-main}.}
\label{tab-degrees}
\end{table}

Rather than giving a complex definition of the graphon $W_0$ at once,
we decided to present the particular details of the structure of $W_0$
together with the decorated constraints fixing the structure of $W_0$ in Sections~\ref{sec-initial}--\ref{sec-white}.
Table~\ref{tab-forcing} gives references to subsections
where the individual tiles of the graphon $W_0$ are considered and
the corresponding decorated constraints are given.

\begin{table}
\begin{center}
\begin{tabular}{|c|cccccccccc|}
\hline
& A & B & C & D & E & F & G & P & Q & R \\
\hline
A & \ref{sec-checkers} & \ref{sec-checkers} & \ref{sec-checkers} & \ref{sec-checkers} &
\ref{sec-checkers} & \ref{sec-checkers} & 
\ref{sec-checkers} & \ref{sec-coordinate} & \ref{sec-degrees} & \ref{sec-degdist} \\
B & \ref{sec-checkers} & \ref{sec-rcheckers} & \ref{sec-rcheckers} &
\ref{sec-rcheckers} & \ref{sec-rcheckers} & \ref{sec-BxN-tile} & \ref{sec-white} & 
\ref{sec-coordinate} & \ref{sec-degrees} & \ref{sec-degdist} \\
C &\ref{sec-rcheckers} & \ref{sec-rcheckers} & \ref{sec-rcheckers} &
\ref{sec-rcheckers} & \ref{sec-rcheckers} & \ref{sec-forc-rec} & \ref{sec-white} & 
\ref{sec-coordinate} & \ref{sec-degrees} & \ref{sec-degdist} \\
D & \ref{sec-checkers} & \ref{sec-rcheckers} & \ref{sec-rcheckers} &
\ref{sec-rcheckers} & \ref{sec-DxM-tile} & \ref{sec-DxN-tile} & \ref{sec-white} & 
\ref{sec-coordinate} & \ref{sec-degrees} & \ref{sec-degdist} \\
E & \ref{sec-checkers} & \ref{sec-rcheckers} & \ref{sec-rcheckers} & \ref{sec-DxM-tile} & 
 \ref{sec-dyadic-sq} & \ref{sec-dyadic-sq} & 
\ref{sec-ref-dyadic-sq} & \ref{sec-coordinate} & \ref{sec-degrees} & \ref{sec-degdist} \\
F & \ref{sec-checkers} & \ref{sec-BxN-tile} & \ref{sec-forc-rec} & \ref{sec-DxN-tile} & 
 \ref{sec-dyadic-sq} & 
\ref{sec-dyadic-sq} & \ref{sec-ref-dyadic-sq} & \ref{sec-coordinate} & \ref{sec-degrees} & 
\ref{sec-degdist} \\
G & \ref{sec-checkers} & \ref{sec-white} & \ref{sec-white} & \ref{sec-white} &
 \ref{sec-ref-dyadic-sq} & \ref{sec-ref-dyadic-sq} & 
\ref{sec-rec-tile} & \ref{sec-coordinate} & \ref{sec-degrees} & \ref{sec-degdist} \\
P & \ref{sec-coordinate} & \ref{sec-coordinate} & \ref{sec-coordinate} & \ref{sec-coordinate} &
\ref{sec-coordinate} & \ref{sec-coordinate} & \ref{sec-coordinate} & \ref{sec-coordinate} & 
\ref{sec-degrees} & \ref{sec-degdist} \\
Q & \ref{sec-degrees} & \ref{sec-degrees} & \ref{sec-degrees} & \ref{sec-degrees} & 
\ref{sec-degrees}& \ref{sec-degrees} & \ref{sec-degrees} & \ref{sec-degrees}& 
\ref{sec-white} & \ref{sec-degdist} \\
R & \ref{sec-degdist} & \ref{sec-degdist} & \ref{sec-degdist} & \ref{sec-degdist} & 
\ref{sec-degdist}& \ref{sec-degdist} & \ref{sec-degdist} & \ref{sec-degdist}& 
\ref{sec-degdist} & \ref{sec-degdist} \\
\hline
\end{tabular}
\end{center}
\caption{Subsections where the structure of the tiles are presented and the related decorated constraints then given.}
\label{tab-forcing}
\end{table}

We now start the proof of the finite forcibility of the graphon $W_0$.
Let $W$ be a graphon that satisfies the constraints from Lemma~\ref{lm-partition}
with respect to the sizes and degrees of the parts of $W_0$ and
that satisfies all the decorated constraints given in Sections~\ref{sec-initial}--\ref{sec-white}.
It will be obvious that the graphon $W_0$ also satisfies these constraints.
So, if we show that $W$ is weakly isomorphic to $W_0$,
then we will have established that $W_0$ is finitely forcible.
We will achieve this goal by constructing a measure preserving map $g:[0,1]\to [0,1]$ such that
$W(x,y)=W_0(g(x),g(y))$ for almost every $(x,y)\in [0,1]^2$.

Let $A,\ldots,G,P,Q,R$ be the parts of the graphon $W$.
To make a clear distinction between the parts of $W$ and $W_0$,
we will use $A_0,\ldots,G_0,P_0,Q_0,R_0\subseteq [0,1]$ to denote the subintervals forming the parts of $W_0$.
The Monotone Reordering Theorem~\cite[Proposition A.19]{bib-lovasz-book} implies that,
for every $X\in\{A,\ldots,G,P,Q,R\}$,
there exist a measure preserving map $\varphi_X:X\to [0,|X|)$ and
a non-decreasing function ${\tilde f}_X:[0,|X|)\to\RR$ such that
$${\tilde f}_X(\varphi_X(x))=\deg_W^P(x)=\frac{1}{|P|}\int\limits_P W(x,y)\dd y$$
for almost every $x\in X$.
The function $g$ maps the vertex $x\in X$, $X\in\{A,\ldots,G,P,Q,R\}$, of $W$ to the vertex $\eta_X(\varphi_X(x)/|X|)$
where $\eta_X$ is the bijective linear map from $[0,1)$ to the part $X_0$ of the graphon $W_0$ of
the form $\eta_X(x)=|X_0|\cdot x+c_X$ for some $c_X\in [0,1]$ (we intentionally define $\eta_X$ in this way,
instead of defining $\eta_X$ as a linear measure preserving map from $[0,|X_0|)$ to $X_0$,
since this definition simplifies our exposition later).
In addition, we define a function $f_X:X\to [0,1]$ as $f_X(x)={\tilde f}_X(\varphi_X(x))$ for every $x\in X$.
Clearly, $g$ is a measure preserving map from $[0,1]$ to $[0,1]$;
hence, we ``only'' need to show that $W(x,y)=W_0(g(x),g(y))$ for almost every $(x,y)\in [0,1]^2$.

\subsection{Coordinate system}
\label{sec-coordinate}

In this subsection, we analyze the tile $P\times P$ and the tiles $P\times X$ (and the symmetric tiles $X\times P$) where $X\in\{A,\ldots,G\}$.
The \emph{half-graphon} $W_{\triangle}$ is the graphon such that
$W_{\triangle}(x,y)$ is equal to $1$ if $x+y\ge 1$ and equal to $0$ otherwise;
the half-graphon is finitely forcible as shown in~\cite{bib-diaconis09+,bib-lovasz11+}.
Consider the decorated constraints from Lemma~\ref{lm-subforcing}
forcing the tile $P\times P$ to be weakly isomorphic to the half-graphon.
This implies that ${\tilde f}_P(x)=\varphi_P(x)/|P|$ for every $x\in [0,|P|)$,
where $\varphi_P$ and ${\tilde f}_P$ are the functions from the Monotone Reordering Theorem used to define the function $g$.
Lemma~\ref{lm-subforcing} and the finite forcibility of the half-graphon yield that
$W(x,y)=W_0(g(x),g(y))$ for almost every $(x,y)\in P^2$.

\begin{figure}[htbp]
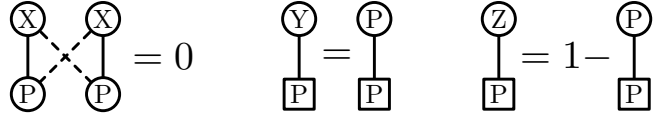

\begin{center}
\epsfbox{ffuniversal-32.mps} \hskip 10mm
\epsfbox{ffuniversal-30.mps} \hskip 10mm
\epsfbox{ffuniversal-31.mps}
\end{center}
\caption{Decorated constraints forcing the tiles $X\times P$ where $X\in\{A,\ldots, G\}$, $Y\in\{E,F,G\}$ and $Z\in\{A,B,C,D\}$.\label{fig-triang-cp}}
\end{figure}

Next consider the decorated constraints depicted in Figure~\ref{fig-triang-cp} and
fix $X\in\{A,\ldots,G\}$.
The first constraint in Figure~\ref{fig-triang-cp} implies
that $W(x,y)\in\{0,1\}$ for almost every $(x,y)\in P\times X$ and
that $N_W^X(x)\sqsubseteq N_W^X(x')$ or $N_W^X(x')\sqsubseteq N_W^X(x)$ for almost every pair $(x,x')\in P^2$.
In addition, the choice of the function $\varphi_X$ implies that
almost every pair $y,y'\in X$ satisfies the following:
if $\varphi_X(y)\le\varphi_X(y')$, then $\deg_W^P(y)\le\deg_W^P(y')$.
Hence, the first constraint and the choice of $\varphi_X$ yield that
for almost every $x\in P$, there exists $t\in [0,|X|]$ such that
the set $N_W^X(x)$ and $\varphi^{-1}_X([0,t))$ differ on a null set.
We conclude that there exists a function $h_X:P\to [0,1]$ such that
it holds for almost every $(x,y)\in P\times X$ that $W(x,y)=1$ if and only if $\varphi_X(y)/|X|\ge 1-h_X(x)$.

If $X\in\{E,F,G\}$, then the second constraint in Figure~\ref{fig-triang-cp} implies that
$\deg^X(x)=|N^X(x)|=\deg^P(x)$ for almost every $x\in P$, i.e., $h_X(x)=f_P(x)$.
Since it holds that $W(x,y)=1$ if and only if $\varphi_X(y)/|X|\ge 1-h_X(x)$ for almost every $(x,y)\in P\times X$,
we obtain that ${\tilde f}_X(y)=\varphi_X(y)/|X|$ for $y\in [0,|X|)$,
$W(x,y)=1$ for almost every $(x,y)\in P\times X$ with $f_P(x)+f_X(y)\ge 1$ and
$W(x,y)=0$ for almost every $(x,y)\in P\times X$ with $f_P(x)+f_X(y)<1$.
It follows that $W(x,y)=W_0(g(x),g(y))$ for almost every $(x,y)\in P\times X$, where $X\in\{E,F,G\}$.
The analogous argument using the third constraint in Figure~\ref{fig-triang-cp} implies that
$\deg^X(x)=|N^X(x)|=|X|-\deg^P(x)$ for almost every $x\in P$,
which yields that $W(x,y)=W_0(g(x),g(y))$ for almost every $(x,y)\in P\times X$, where $X\in\{A,B,C,D\}$.

We conclude this subsection by observing that
$\deg_W^P(x)=f_X(x)$ for almost every $x\in X$, where $X\in\{A,\ldots,G\}\cup\{P\}$.
In particular, we may interpret the relative degree of a vertex with respect to $P$ as its coordinate.
Also observe that $N_W^P(x)\sqsubseteq N_W^P(x')$
for almost every pair $(x,x')\in X\times X$ such that $f_X(x)\le f_X(x')$.

\subsection{Checker tiles}
\label{sec-checkers}

We now consider the tiles $A\times X$ where $X \in \{A, \ldots G\}$.
The argument follows the lines of the analogous argument presented in~\cite{bib-comp, bib-inf, bib-reg},
however, we include the details for completeness.
The \emph{checker graphon} $W_C$ is obtained as follows:
let $I_k = [1 - 2^{-k}, 1 - 2^{-(k+1)})$ for $k\in\NN_0$ and
set $W_C(x, y)$ equal to $1$ if $(x, y) \in \bigcup \limits_{k = 0}^{\infty} I_{k}^2$,
i.e., both $x$ and $y$ belong to the same $I_k$, and equal to $0$ otherwise.
The checker graphon $W_C$ is depicted in Figure~\ref{fig-checkers}.
We remark that we present an iterated version of this definition in Subsection~\ref{sec-rcheckers}.
We set $W_0(\eta_A(x),\eta_X(y))=W_C(x,y)$ for $x,y\in [0,1)^2$ where $X \in \{A, \ldots G\}$.

\begin{figure}[htbp]
\begin{center}
\epsfbox{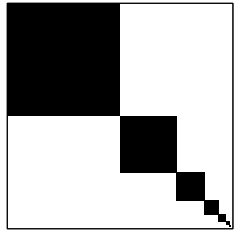}
\end{center}
\caption{The checker graphon $W_C$.\label{fig-checkers}}
\end{figure}

Consider the decorated constraints in Figure~\ref{fig-checkers-const},
which we claim to force the structure of the tile $A\times A$.
The argument follows the lines of the analogous argument given in~\cite[Section 5]{bib-comp},
so we sketch the main steps here and refer the reader for full details to~\cite{bib-comp}.
The first constraint in Figure~\ref{fig-checkers-const} implies that
there exists a collection $\JJ_A$ of disjoint measurable non-null subsets of $A$ such that
the following holds for almost every $(x,y)\in A\times A$:
$W(x,y)=1$ if and only if $x$ and $y$ belong to the same set $J\in\JJ_A$, and $W(x,y)=0$ otherwise.

\begin{figure}[htbp]
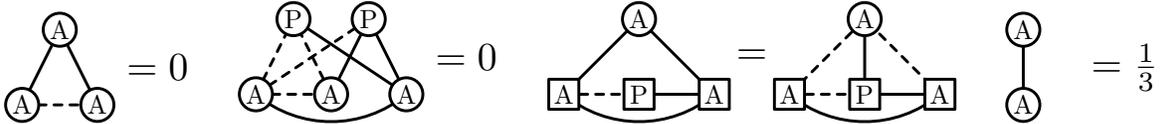

\begin{center}
\epsfbox{ffuniversal-7.mps} \hskip 5mm
\epsfbox{ffuniversal-8.mps} \hskip 5mm
\epsfbox{ffuniversal-9.mps} \hskip 5mm
\epsfbox{ffuniversal-10.mps}
\end{center}
\caption{The decorated constraints forcing the structure of the tile $A\times A$.\label{fig-checkers-const}}
\end{figure}

The second constraint in Figure~\ref{fig-checkers-const} implies that almost every triple $(x,x',x'')\in A^3$ satisfies that
if $x$ and $x''$ belong to the same set $J\in\JJ_A$ and $f_A(x)<f_A(x')<f_A(x'')$,
then $x'$ also belongs to the set $J$ (since $x$ and $x'$ cannot be non-adjacent).
This implies that for every $J\in\JJ_A$,
there exists an open interval $J'$ such that $J$ and $f_A^{-1}(J')$ differ on a null set.
Let $\JJ'_A$ be the collection of these open intervals for different sets $J\in\JJ_A$;
since $f_A$ is a measure preserving map and the sets in $\JJ_A$ are disjoint,
the intervals in $\JJ'_A$ must be disjoint.

The third constraint in Figure~\ref{fig-checkers-const} implies that almost every pair $(x,x')\in A^2$ satisfies that
if $x$ and $x'$ belong to the same set $J\in\JJ_A$ and $f_A(x)<f_A(x')$,
then $|N_W^A(x)\cap N_W^A(x')|=|J|$ is the measure of the set $Y$ of the points $y\in A$
such that $y\not\in J$ and $f_A(y)>f_P(x'')$ for almost every $x''\in P$ with $f_A(x)<f_P(x'')<f_A(x')$.
Observe that if $J$ is fixed and $J=f_A^{-1}(J')$ for $J'\in\JJ'_A$,
then the set $Y$ differs from $f_A^{-1}([\sup J',1))$ on a null set.
It follows that the measure $|J|=|J'|$ is equal to $1-\sup J'$.
Hence, each interval in $\JJ'_A$ is of the form $(1-2\gamma,1-\gamma)$ for some $\gamma\in (0,1/2]$;
let $\Gamma$ be the set of all the values of $\gamma$ for that there is a corresponding interval in $\JJ'_A$.
Note that if $\gamma\in \Gamma$, then $\Gamma\cap(\gamma/2,\gamma)=\emptyset$,
which implies in particular that the set $\Gamma$ is countable.
Let $\gamma_k$ be the $k$-th largest value in the set $\Gamma$ and
in case that $\Gamma$ is finite, set $\gamma_k=0$ for $k>|\Gamma|$.
It follows that
$$\frac{1}{|A|^2}\int\limits_{A\times A}W(x,y)\dd x\dd y=\sum_{J'\in\JJ'_A}\left(\sup J'-\inf J'\right)^2=\sum_{k\in\NN}\gamma_k^2\;\mbox{.}$$
The last constraint in Figure~\ref{fig-checkers-const} implies that the integral on the left hand side of the above equality
is equal to $1/3$, which is possible only if $\gamma_k=2^{-k}$ for every $k\in\NN$.
It follows that $W(x,y)=W_0(g(x),g(y))$ for almost every $(x, y) \in A^2$.

\begin{figure}[htbp]
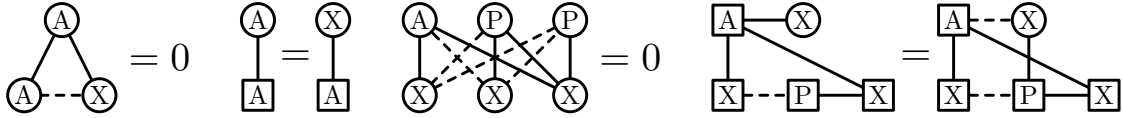

\begin{center}
\epsfbox{ffuniversal-11.mps} \hskip 5mm
\epsfbox{ffuniversal-12.mps} \hskip 5mm
\epsfbox{ffuniversal-13.mps} \hskip 5mm
\epsfbox{ffuniversal-14.mps}
\end{center}
\caption{The decorated constraints forcing the structure of the tiles $A\times X$ where $X\in\{B,\ldots,G\}$.\label{fig-cp-checkers}}
\end{figure}

We now consider the decorated constraints from Figure~\ref{fig-cp-checkers}.
Fix $X \in \{B, \ldots,G\}$.
The first constraint in Figure~\ref{fig-cp-checkers} implies that for every $J\in\JJ_A$,
there exists a measurable set $Z(J)\subseteq X$ such that
the following holds for almost every pair $(x,y)\in A\times X$:
$W(x,y)=1$ if $x\in J$ and $y\in Z(J)$, and $W(x,y)=0$ otherwise.
Note that the sets $Z(J)$ need not be disjoint.
The second constraint in Figure~\ref{fig-cp-checkers} yields that
$\deg_W^A(x)=\deg_W^X(x)$ for almost every $x \in A$,
which implies that the sets $J$ and $Z(J)$ have the same measure.
The third constraint implies that the following holds for almost every triple $(y,y',y'')\in X^3$:
if $f_P(y)<f_P(y')<f_P(y'')$, $y\in Z(J)$ and $y''\in Z(J)$, then $y'\in Z(J)$.
Consequently, for every $Z(J)$, there exists an open interval $Z'(J)$ such that
$Z(J)$ differs from the set $g_X^{-1}(Z'(J))$ on a null set.
Finally, the last constraint in Figure~\ref{fig-cp-checkers} yields that
the following holds for almost every $x\in J$:
the measure of $N_W^X(x)=Z(J)$,
which is $|Z(J)|=|Z'(J)|$, is equal to the measure of the set containing all $y\not\in Z(J)$ with $f_X(y)\ge\sup Z'(J)$.
It follows that the interval $Z'(J)$ is equal to $(1-2\gamma,\gamma)$ for some $\gamma\in (0,|X|/2]$.
Since the measures of $J$ and $Z'(J)$ are the same,
it must hold that $Z'(J)=J'$ where $J'\in\JJ'_A$ is the interval corresponding to $J$.
It follows that $W(x,y)=W_0(g(x),g(y))$ for almost every $(x, y) \in A\times X$.

\subsection{Iterated checker tiles}
\label{sec-rcheckers}

The checker graphon $W_C$ represents a large graph formed by disjoint complete graphs
on the $1/2,1/4,1/8,\ldots$ fractions of its vertices.
We now present a family of iterated checker graphons.
Informally speaking, we start with the checker graphon $W_C$ and
at each iteration, we paste a scaled copy of $W_C$ on each clique of the current graphon.
The formal definition is as follows.
Fix $k\in\NN_0$.
If $k=0$, define $I_{j_0}$, $j_0\in\NN_0$, to be the interval
$$I_{j_0}=\left[1-2^{-j_0},1-2^{-j_0-1}\right).$$
If $k>0$, we define $I_{j_0,\ldots,j_k}$ for $(j_0,\ldots,j_k)\in\NN_0^k$ as
$$I_{j_0,\ldots,j_k}=\left[\sup I_{j_0,\ldots,j_{k-1}}-2^{-j_k}|I_{j_0,\ldots,j_{k-1}}|,\sup I_{j_0,\ldots,j_{k-1}}-2^{-j_k-1}|I_{j_0,\ldots,j_{k-1}}|\right).$$
The \emph{$k$-iterated checker graphon} $W_C^k$ is then defined as follows:
$W_C^k(x,y)$ is equal to $1$ if there exists a $(k+1)$-tuple $(j_0,\ldots,j_k)\in\NN_0^k$ such that
both $x$ and $y$ belong to the interval $I_{j_0,\ldots,j_k}$, and
it is equal to $0$ otherwise.
The iterated checker graphons $W_C^0$, $W_C^1$ and $W_C^2$ are depicted in Figure~\ref{fig-iterated-checker}.
Note that $W_C^0=W_C$ and the definition of $I_{j_0}$ coincides with that given in Subsection~\ref{sec-checkers}.
We will also refer to an interval $I_{j_0,\ldots,j_k}$ as to a \emph{$k$-iterated binary interval}.

\begin{figure}
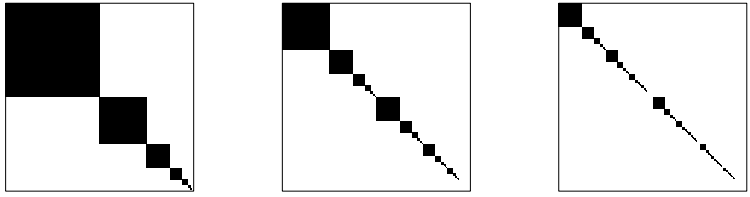

	\begin{center}
		\epsfbox{ffuniversal-300.mps} \hskip 10mm
		\epsfbox{ffuniversal-301.mps} \hskip 10mm
		\epsfbox{ffuniversal-302.mps}
	\end{center}
\caption{The iterated checker graphons $W_C^0$, $W_C^1$ and $W_C^2$.\label{fig-iterated-checker}}
\end{figure}

For $X\in\{B,C\}$ and $Y\in\{X,\ldots,E\}$, we set
$$W_0(\eta_X(x),\eta_Y(y))=\left\{
  \begin{array}{cl}
  W_C^1(x,y) & \mbox{if $X=B$, and}\\
  W_C^2(x,y) & \mbox{if $X=C$}\\
  \end{array}
  \right.$$
for all $x,y\in [0,1)^2$.
We also set the tile $D\times D$ to be such that 
$$W_0(\eta_D(x),\eta_D(y))=W_C^3(x,y)$$
for all $x,y\in [0,1)^2$.
This also defines the values of $W_0$ in the symmetric tiles,
i.e., the values for the tile $X\times Y$ determine the values for the tile $Y\times X$.

\begin{figure}[htbp]
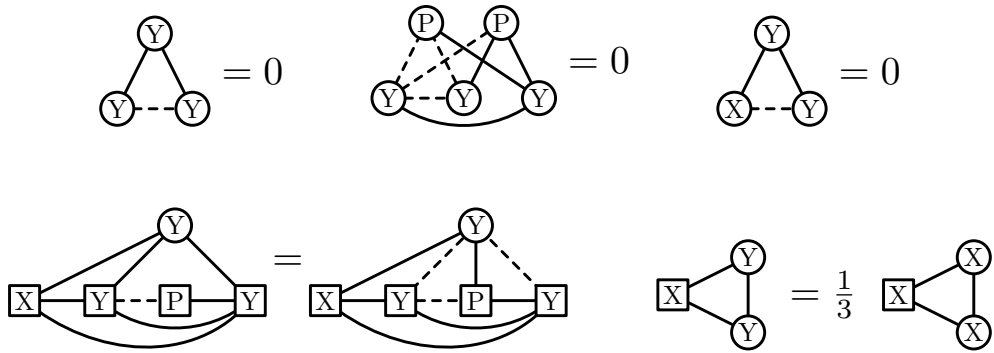

\begin{center}
\epsfbox{ffuniversal-77.mps} \hskip 10mm
\epsfbox{ffuniversal-78.mps} \hskip 10mm
\epsfbox{ffuniversal-67.mps} \vskip 10mm
\epsfbox{ffuniversal-68.mps} \hskip 10mm
\epsfbox{ffuniversal-69.mps}
\end{center}
\caption{The decorated constraints forcing the structure of the tiles $B^2$, $C^2$, and $D^2$,
         where $(X,Y)\in\{(A,B),(B,C),(C,D)\}$.
         \label{fig-rec-checkers}}
\end{figure}

\begin{figure}[htbp]
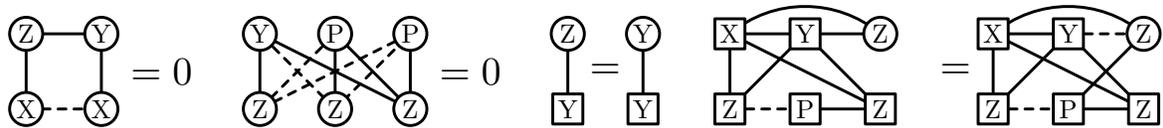

\begin{center}
\epsfbox{ffuniversal-70.mps} \hskip 5mm
\epsfbox{ffuniversal-72.mps} \hskip 5mm
\epsfbox{ffuniversal-71.mps} \hskip 5mm
\epsfbox{ffuniversal-73.mps}
\end{center}
\caption{The decorated constraints forcing the structure of the iterated checker graphons on the non-diagonal tiles,
         where $(X,Y)\in\{(A,B),(B,C)\}$ and $Z\in\{C,D,E,F\}$ if $X=A$ and $Z\in\{D,E,F\}$ if $X=B$.\label{fig-rec-cp-checkers}}
\end{figure}

Consider the decorated constraints depicted in Figures~\ref{fig-rec-checkers} and~\ref{fig-rec-cp-checkers}.
We first analyze the structure of the tile $B\times B$, then all the tiles $B\times Y$, $Y\in\{B,\ldots,E\}$,
then the tile $C\times C$, then all the tiles $C\times Y$, $Y\in\{C,\ldots,E\}$, before finishing with the tile $D\times D$.
Fix $(X,Y)$ to be one of the pairs $(A,B),(B,C)$ or $(C,D)$.
We assume that $W(x,y)=W_0(g(x),g(y))$ for almost every $(x,y)\in X\times X$ and almost every $(x,y)\in X\times Y$, and
our goal is to show that $W(x,y)=W_0(g(x),g(y))$ for almost every $(x,y)\in Y\times Y$.

The first two constraints on the first line in Figure~\ref{fig-rec-checkers} imply that
there exists a collection $\JJ'_Y$ of disjoint open intervals such that
the following holds for almost every $(x,y)\in Y^2$:
$W(x,y)$ is equal to $1$ if and only if
$f_Y(x)$ and $f_Y(y)$ belong to the same interval $J'\in\JJ'_Y$, and
it is equal to $0$ otherwise.
The third constraint on the first line in Figure~\ref{fig-rec-checkers} yields that
each interval in $\JJ'_Y$ is a subinterval of an interval in $\JJ'_X$.

The first constraint on the second line in Figure~\ref{fig-rec-checkers} yields that
the following holds for almost every triple $(x,y,y')\in X\times Y\times Y$ such that
$f_Y(y)$ and $f_Y(y')$ are from the same interval $J'_Y\in\JJ'_Y$ and
$f_X(x)$ is from the interval $J'_X\in\JJ'_X$ that is a superinterval of $J'_Y$:
the measure of $J'_Y$ (which is equal to the left hand side of the equality) is the same as
the measure of the set of all $y''$ such that $f_Y(y'')\in J'_X$ and $f_Y(y'')>\sup J'_Y$ (which is equal to the right hand side).
It follows that
$$J'_Y=(\sup J'_X-2\gamma,\sup J'_X-\gamma)$$
for some $\gamma\in (0,|J'_X|/2]$.
The very last constraint in Figure~\ref{fig-rec-checkers} yields for every $J'_X\in\JJ'_X$ that
$$\sum_{J'_Y\in\JJ'_Y,J'_Y\subseteq J'_X}|J'_Y|^2=\frac{1}{3}|J'_X|^2.$$
However, this is only possible if the set $\JJ'_Y$ contains all intervals of the form
$(\sup J'_X-2\gamma,\sup J'_X-\gamma)$ for every $J'_X\in\JJ'_X$ and every $\gamma=|J'_X|\cdot 2^{-i}$, $i\in\NN$.
It follows that $W(x,y)=W_0(g(x),g(y))$ for almost every $(x,y)\in Y\times Y$.

We continue to fix a pair $(X,Y)\in\{(A,B),(B,C)\}$, but in addition we now 
fix $Z\in\{Y,\ldots,E\}\setminus\{Y\}$ where $Y\in\{B,C\}$.
Our next goal is to show that $W(y,z)=W_0(g(y),g(z))$ for almost every $(y,z)\in Y\times Z$,
which is achieved using the decorated constrains given in Figure~\ref{fig-rec-cp-checkers}.
The first constraint in Figure~\ref{fig-rec-cp-checkers} implies that
it holds for almost every $y\in Y$ that
$f_Z(N_W^Z(y))\sqsubseteq J'_X$ where $J'_X$ is the unique interval of $\JJ'_X$ containing $f_Y(y)$.
The second constraint in Figure~\ref{fig-rec-cp-checkers} yields that for almost every $y\in Y$,
there exists an interval $J_y$ such that $N_W^Z(y)$ and $f_Z^{-1}(J_y)$ differ on a null set,
$W(y,z)=1$ for almost every $z\in f_Z^{-1}(J_y)$, and
$W(y,z)=0$ for almost every $z\in Z\setminus f_Z^{-1}(J_y)$.
The third constraint yields that
$\deg_W^Y(y)=\deg_W^Z(y)$ for almost every $y\in Y$,
i.e., the measure of $J_y$ is the same as the measure of the interval in $\JJ_Y$ containing $f_Y(y)$.

Finally, the last constraint in Figure~\ref{fig-rec-cp-checkers} implies that
almost every quadruple $x\in X$, $y\in Y$, $z,z'\in Z$ such that
$f_Z(z)<f_Z(z')$, $f_Z(z)$ and $f_Z(z')$ belong to the interval $J_y$,
which is a subinterval of $J'_X\in\JJ'_X$ with $f_X(x)\in J'_X$,
satisfies that
the measure of $N_W^Z(y)$ (note that $N_W^Z(y)$ is a subset of $f_Z^{-1}(J'_X)$) and
the measure of all $z'\in f_Z^{-1}(J'_X)\setminus N_W^Z(y)$ with $f_Z(z')>\sup J_y$ are equal.
In particular, the interval $J_y$ is of the form
$(\sup J'_X-2\gamma,\sup J'_X-\gamma)$ for almost every $y\in Y$,
where $J'_X$ is the unique interval of $\JJ'_X$ containing $f_Y(y)$.
Hence, the interval $J_y$ is equal to the interval in $\JJ'_Y$ containing $f_Y(y)$ for almost every $y\in Y$.
It follows that $W(y,z)=W_0(g(y),g(z))$ for almost every $(y,z)\in Y\times Z$.

\section{Encoding the target graphon}

In this section, we describe how the densities in dyadic squares of the graphon $W_F$
are wired in a single binary sequence, which will be encoded in the tile $B\times F$.
To achieve this, we need to fix a mapping $\varphi$ from $\NN_{0}^4$ to $\NN_{0}$.
Let us define this mapping as follows. 
The $4$-tuples $(a,b,c,d)$ with the same sum $s=a+b+c+d$ of their entries
are injectively mapped to the numbers between ${s+3\choose 4}$ and ${s+4\choose 4}-1$
in the lexicographic order. For example, $\varphi(0,0,0,1)=1$ and $\varphi(0,1,0,0)=3$.

\subsection{Encoding dyadic square densities}
\label{sec-BxN-tile}

The tile $B\times F$ encodes the edge densities on all dyadic squares of $W_F$.
Let $I^d(s)$ be the interval $\left[\frac{s}{2^d},\frac{s+1}{2^d}\right)$, and
define for $d,s,t\in\NN_0$ the value $\delta(d,s,t)$ as
$$\delta(d,s,t)=2^{2d}\cdot\int_{I^d(s)\times I^d(t)} W_F(x,y) \dd x\dd y$$
if $0\le s,t\le 2^d-1$, and $\delta(d,s,t)=0$, otherwise.
If $W_F$ is the one graphon, i.e., $W_F$ is equal to $1$ almost everywhere, we fix $r=1$.
Otherwise, we fix $r\in [0,1)$ to be the unique real satisfying that
\begin{equation}
\delta(d,s,t)=\sum_{p=0}^\infty 2^{-p}r_{\varphi(d,s,t,p)+1}\;\mbox{, and}\label{eq-R}
\end{equation}
that for all $d,s,t\in\NN_0$,
the value of $r_{\varphi(d,s,t,p)+1}$ is equal to zero for infinitely many $p\in\NN_0$,
where $r_k$ is the $k$-th bit in the standard binary representation of $r$ (with the first bit
following immediately the decimal point).
The standard binary representation is the unique representation with infinitely many digits equal to zero
If $W_F$ is the one graphon, we set $r_k=1$ for every $k\in\NN$.
Observe that $r$ is not a multiple of an inverse power of two
unless $W_F$ is the zero graphon or the one graphon ($r\in\{0,1\}$ in these two cases).

\begin{figure}[htbp]
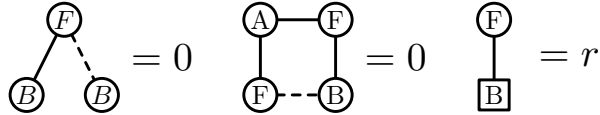

	\begin{center}
		\epsfbox{ffuniversal-21.mps} \hskip 5mm
		\epsfbox{ffuniversal-22.mps} \hskip 5mm
		\epsfbox{ffuniversal-23.mps} 
	\end{center}
	\caption{The decorated constraints forcing the structure of the tile $B\times F$.\label{fig:inputTile}}
\end{figure}

We now define $W_0(\eta_B(x),\eta_F(y))=r_{k+1}$ for $x \in [0,1]$ and $y \in I_{k}$, $k\in\NN_0$, and
force the corresponding structure of the tile $B\times F$.
Consider the decorated constraints depicted in Figure~\ref{fig:inputTile}.
The first constraint implies that $\deg_{W}^{B}(x) \in \{0,1\}$ for almost every $x \in F$. 
In particular, $W$ is $\{0,1\}$-valued almost everywhere on $B\times F$. 
The second constraint implies that for every $k \in \NN_0$ and 
for almost every $x,x' \in f_{F}^{-1}(I_k)$, $\deg_{W}^{B}(x)= \deg_{W}^{B}(x')$.
Let $r'_k$ be the common degree $\deg_{W}^{B}(x)$ of the vertices $x\in f_{F}^{-1}(I_{k-1})$, $k\in\NN$.
The last constraint in the figure implies that 
$$\sum_{k\in\NN}2^{-k}r_k=\sum_{k\in\NN}2^{-k}r'_k\;\mbox{.}$$
Since $r$ is not a non-zero multiple of an inverse power of two unless $r\in\{0,1\}$,
it follows that $r_k=r'_k$ for all $k\in\NN$.
If $r\in\{0,1\}$, it follows that $r_k=r'_k=r$ trivially.
We conclude that $W(x,y) = W_0(g(x), g(y))$ for almost every $(x, y) \in B \times F$.

\subsection{Matching tile}
\label{sec-DxN-tile}

In this subsection, we introduce and analyze the tile $D\times F$.
This tile is supposed to link the $4$-fold indexing to linear indexing.
Formally,
we define $W_0(\eta_D(x),\eta_F(y))$ to be equal to $1$ if $x \in I_{a,b,c,d}$ and $y \in I_{\varphi(a,b,c,d)}$ for some $(a,b,c,d)\in\NN_0^4$ and to be equal to $0$, otherwise.

\begin{figure}[htbp]
	\begin{center}
		\epsfbox{ffuniversal-17.mps} \hskip 5mm
		\epsfbox{ffuniversal-79.mps} \hskip 5mm
		\epsfbox{ffuniversal-29.mps} \hskip 5mm
		\epsfbox{ffuniversal-16.mps} \vskip 10mm
		\epsfbox{ffuniversal-15.mps} \vskip 10mm
		\epsfbox{ffuniversal-18.mps} \vskip 10mm
		\epsfbox{ffuniversal-19.mps} \vskip 10mm
		\epsfbox{ffuniversal-20.mps}
	\end{center}
	\caption{The decorated constraints forcing the structure of the tile $D\times F$.\label{fig:dxn}}
\end{figure}

Consider the decorated constraints in Figure~\ref{fig:dxn}. 
The first constraint implies that $W$ is $\{0,1\}$-valued almost everywhere in $D\times F$ and 
that for almost every $x\in D$, it holds that $N_{W}^{F}(x) = \cup_{k\in K_x}f_{F}^{-1}(I_k)$ up to a null set
for some $K_x \subseteq \NN_0$.
The second constraint implies that for almost every vertex of $D$, the set $K_x$ has cardinality 0 or 1.
The third constraint yields that for every $(a,b,c,d)\in\NN_0^4$,
the set $K_x$ is the same for almost all $x\in D$ with $f_D(x)\in I_{a,b,c,d}$.
Finally,
the last constraint in the first line implies that 
the sets $K_x$ and $K_y$ are disjoint for almost all $x,y\in D$ with $f_D(x)$ and $f_D(y)$
from different $3$-iterated binary intervals.

Let $\tau(a,b,c,d)$ be the common degree $\deg_{W}^{F}(x)$ of vertices $x \in f_{D}^{-1}(I_{a,b,c,d})$.
If $K_x$ is empty for almost all $x \in f_{D}^{-1}(I_{a,b,c,d})$, then $\tau(a,b,c,d)=0$;
otherwise, $\tau(a,b,c,d)$ is $2^{-k-1}$, where $k$ is the unique integer contained in $K_x$
for almost all $x \in f_{D}^{-1}(I_{a,b,c,d})$.
Note that the non-zero values of $\tau(a,b,c,d)$ are distinct for distinct $(a,b,c,d)\in\NN_0^4$.
Observe that the edge density in the tile $D\times F$ is the following:
$$\int_{D\times F} W(x,y) \dd x \dd y=  \sum_{(a,b,c,d) \in \NN_{0}^4} |I_{a,b,c,d}| \tau(a,b,c,d)
= \sum_{s \in \NN_{0}} 2^{-(s+4)}\sum_{\substack{(a,b,c,d) \in \NN_{0}^4 \\a+b+c+d = s}} \tau(a,b,c,d).$$
The constraint in the second line in Figure~\ref{fig:dxn} yields the following:
$$\sum_{s \in \NN_{0}} 2^{-(s+4)}\sum_{\substack{(a,b,c,d) \in \NN_{0}^4 \\a+b+c+d = s}} \tau(a,b,c,d) =
  \sum_{s \in \NN_{0}} 2^{-(s+4)}\sum_{\substack{(a,b,c,d) \in \NN_{0}^4 \\a+b+c+d = s}} 2^{-\varphi(a,b,c,d)-1}\;\mbox{.}$$
Since the non-zero values of $\tau(a,b,c,d)$ are mutually distinct,
this equality can hold only if
$$\{\tau(a,b,c,d)\mbox{ s.t. } a+b+c+d=s\}=\{2^{-\varphi(a,b,c,d)-1}\mbox{ s.t. } a+b+c+d=s\}$$
for every $s\in\NN_0$.

The constraint in the third line in Figure~\ref{fig:dxn} implies that
$$\sum_{(a,b,c,d) \in \NN_{0}^4} 2^{-a-b-c-d-4}\cdot 2^{-a-b-c-3}\cdot \tau(a,b,c,d)=
  \hskip -2ex
  \sum_{(a,b,c,d) \in \NN_{0}^4} 2^{-a-b-c-d-4}\cdot 2^{-a-b-c-3}\cdot 2^{-\varphi(a,b,c,d)-1}\;\mbox{.}$$
Since it holds for every $s\in\NN_0$ that
$$\{\tau(a,b,c,d)\mbox{ s.t. } a+b+c+d=s\}=\{2^{-\varphi(a,b,c,d)-1}\mbox{ s.t. } a+b+c+d=s\}\;\mbox{,}$$
we get that the following holds for all $d\in\NN_0$ and $s\in\NN_0$:
$$\{\tau(a,b,c,d)\mbox{ s.t. } a+b+c=s\}=\{2^{-\varphi(a,b,c,d)-1}\mbox{ s.t. } a+b+c=s\}\;\mbox{.}$$

Similarly, the constraint in the fourth line implies that
$$\sum_{(a,b,c,d) \in \NN_{0}^4} 2^{-a-b-c-d-4}\cdot 2^{-a-b-2}\cdot \tau(a,b,c,d)=
  \sum_{(a,b,c,d) \in \NN_{0}^4} 2^{-a-b-c-d-4}\cdot 2^{-a-b-2}\cdot 2^{-\varphi(a,b,c,d)-1}\;\mbox{,}$$
which yields that it holds for all $c,d,s\in\NN_0$ that
$$\{\tau(a,b,c,d)\mbox{ s.t. } a+b=s\}=\{2^{-\varphi(a,b,c,d)-1}\mbox{ s.t. } a+b=s\}\;\mbox{.}$$
Finally, the constraint in the fifth line implies that
$$\sum_{(a,b,c,d) \in \NN_{0}^4} 2^{-a-b-c-d-4}\cdot 2^{-a-1}\cdot \tau(a,b,c,d)=
  \sum_{(a,b,c,d) \in \NN_{0}^4} 2^{-a-b-c-d-4}\cdot 2^{-a-1}\cdot 2^{-\varphi(a,b,c,d)-1}\;\mbox{,}$$
which implies that $\tau(a,b,c,d)=2^{-\varphi(a,b,c,d)-1}$ for all $a,b,c,d\in\NN_0$.
It follows that $W(x,y) = W_0(g(x), g(y))$ for almost every $(x, y) \in D \times F$.

\subsection{Collating dyadic square densities}
\label{sec-DxM-tile}
The tile $D\times E$ is designed to group the values of $\delta(d,s,t)$.
We set $W_0(\eta_D(x),\eta_E(y))=r_{\varphi(d,s,t,p)+1}$ for all $x\in I_{d,s,t,p}$, $y \in I_{d,s,t}$ and
$(d,s,t,p)\in\NN_0^4$, and we set $W_0(\eta_D(x),\eta_E(y))$ to be zero elsewhere.
An example of a tile with this structure is depicted in Figure~\ref{fig-ex-dens-limit}.
Note that the density of the square $\eta_D(I_{d,s,t})\times \eta_E(I_{d,s,t})$ is equal to $\delta(d,s,t)$.

\begin{figure}[htbp]
	\begin{center}
		\epsfbox{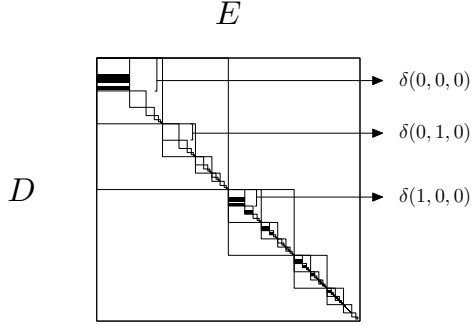} 
	\end{center}
	\caption{An example of the $D\times E$ tile.\label{fig-ex-dens-limit}}
\end{figure}

Consider the decorated constraints depicted in the Figure~\ref{fig:limitTile}.
The first constraint implies that $W(x,y)=0$ for almost every $(x,y)$ such that
$x \in f_{D}^{-1}(I_{d,s,t})$, $y \in f_{E}^{-1}(I_{d',s',t'})$ and $(d,s,t)\neq (d',s',t')$.
The second constraint yields that for almost every $x\in D$ such that $x\in f_D^{-1}(I_{d,s,t})$,
$\deg_{W}^{E}(x)$ is either $0$ or $2^{-d-s-t-3}$.
In particular, $W(x,y)\in\{0,1\}$ for almost every $(x,y)\in D\times E$.

\begin{figure}[htbp]
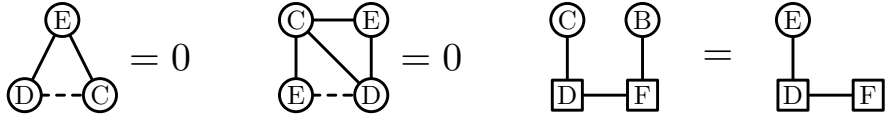

	\begin{center}
		\epsfbox{ffuniversal-25.mps} \hskip 10mm
		\epsfbox{ffuniversal-27.mps} \hskip 10mm
		\epsfbox{ffuniversal-28.mps} 
	\end{center}
	\caption{The decorated constraints forcing the structure of the tile $D\times E$.\label{fig:limitTile}}
\end{figure}

We now analyze the last decorated constraint depicted in the Figure~\ref{fig:limitTile}.
This constraint implies that the following holds for almost every choice of a $D$-root $x$ and an $F$-root $y$ such that
$f_D(x)\in I_{d,s,t,p}$ and $f_F(y)\in I_{\varphi(d,s,t,p)}$:
$$ 2^{-d-s-t-3}\cdot r_{\varphi(d,s,t,p)+1}=\deg_{W}^{E}(x)\;\mbox{.}$$
It follows that $W(x,y) = W_0(g(x), g(y))$ for almost every $(x, y) \in D \times E$.

\section{Forcing the target graphon}

In this section, we force the densities in each dyadic square of the tile $G\times G$ to be as in the graphon $W_F$ and
we argue that the graphon inside the tile is the graphon $W_F$.
To achieve this, we first need to set up some auxiliary structures.

\subsection{Dyadic square indices}
\label{sec-dyadic-sq}

We start with the tiles $E\times E$, $E\times F$ and $F\times F$,
which represent splitting the $0$-iterated binary interval $I_k$ into $2^k$ and $2^{2k}$ equal length parts.
Formally, $W_0(\eta_E(x),\eta_E(y))$ is equal to $1$ for $x,y\in [0,1)$
if $x$ and $y$ belong to the same $0$-iterated binary interval $I_k$ and
$$\left\lfloor\frac{x-\min I_k}{|I_k|}\cdot 2^k\right\rfloor=
  \left\lfloor\frac{y-\min I_k}{|I_k|}\cdot 2^k\right\rfloor\;\mbox{,}$$
and it is equal to $0$ otherwise.
Similarly, $W_0(\eta_F(x),\eta_F(y))$ is equal to $1$ for $x,y\in [0,1)$
if $x$ and $y$ belong to the same $0$-iterated binary interval $I_k$ and
$$\left\lfloor\frac{x-\min I_k}{|I_k|}\cdot 2^{2k}\right\rfloor=
  \left\lfloor\frac{y-\min I_k}{|I_k|}\cdot 2^{2k}\right\rfloor\;\mbox{,}$$
and it is equal to $0$ otherwise.
An illustration can be found in Figure~\ref{fig-example-enum-xy}.
Finally, we set $W_0(\eta_E(x),\eta_F(y))=W_0(\eta_E(x),\eta_E(y))$ for all $x,y\in [0,1)$.

\begin{figure}[htbp]
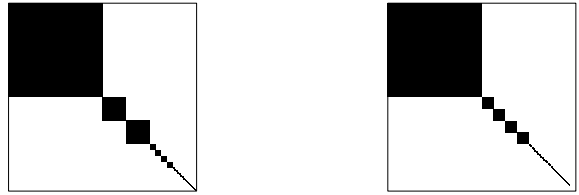

	\begin{center}
\epsfbox{ffuniversal-305.mps}\hskip 25mm
\epsfbox{ffuniversal-306.mps}
\end{center}
\caption{Representation of the tiles $E\times E$ and $F\times F$.\label{fig-example-enum-xy}}
\end{figure}

\begin{figure}[htbp]
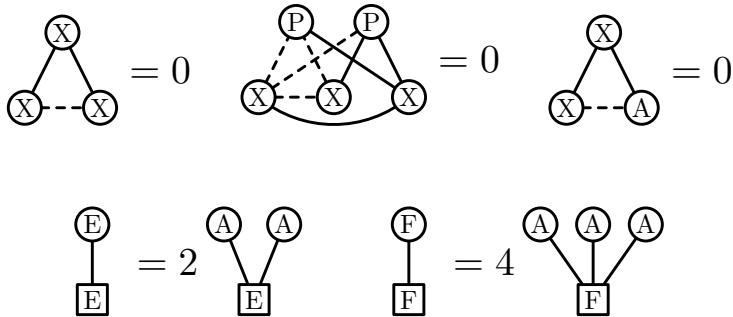

\begin{center}
\epsfbox{ffuniversal-237.mps} \hskip 5mm
\epsfbox{ffuniversal-238.mps} \hskip 5mm
\epsfbox{ffuniversal-34.mps} \vskip 10mm
\epsfbox{ffuniversal-35.mps} \hskip 10mm
\epsfbox{ffuniversal-37.mps} 
\end{center}
\caption{The decorated constraints forcing the tiles $E\times E$ and $F\times F$, where $X \in \{E, F\}$.\label{fig-benum-xy}}
\end{figure}

Fix $X\in\{E,F\}$ and consider the decorated constraints given in Figure~\ref{fig-benum-xy}.
The three constraints on the first line in Figure~\ref{fig-benum-xy}
imply that $W(x,y)\in\{0,1\}$ for almost every pair $(x,y)\in X\times X$ and that
there exists a collection of disjoint open intervals $\JJ_X$, which are subintervals of $0$-iterated binary intervals $I_k$, such that
$W(x,y)=1$ if and only if $f_X(x)$ and $f_X(y)$ belong to the same interval $J\in\JJ_X$ (except for a subset of $X\times X$ of measure zero).

If $X=E$, then the first constraint on the second line in Figure~\ref{fig-benum-xy} implies that
$\deg_W^E(x)=2^{-2k-1}$ for almost every $x\in f_E^{-1}(J_k)$,
i.e., if $J\in\JJ_E$ and $J\subseteq I_k$, then $|J|=2^{-k}|I_k|$.
Hence, the set $\JJ_E$ is formed precisely by the intervals
$$\left(\min I_k+\frac{\ell-1}{2^k}|I_k|, \min I_k+\frac{\ell}{2^k}|I_k|\right)$$
for $k\in\NN_0$ and $\ell\in [2^k]$.
Hence, $W(x,y)=W_0(g(x),g(y))$ for almost every $(x,y)\in E\times E$.
The analogous argument using the last constraint on the second line in Figure~\ref{fig-benum-xy} gives that
$\JJ_F$ is formed precisely by the intervals
$$\left(\min I_k+\frac{\ell-1}{2^{2k}}|I_k|, \min I_k+\frac{\ell}{2^{2k}}|I_k|\right)$$
for $k\in\NN_0$ and $\ell\in [2^{2k}]$,
which leads to the conclusion that $W(x,y)=W_0(g(x),g(y))$ for almost every $(x,y)\in F\times F$.

\begin{figure}[htbp]
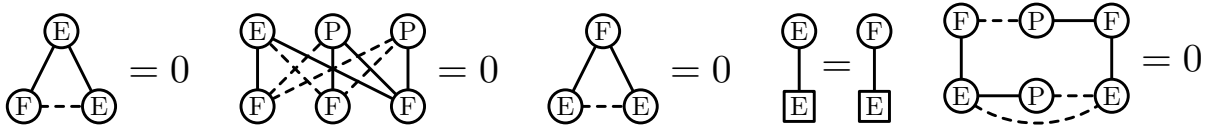

\begin{center}
\epsfbox{ffuniversal-39.mps} \hskip 5mm
\epsfbox{ffuniversal-40.mps} \hskip 5mm
\epsfbox{ffuniversal-311.mps} \hskip 5mm
\epsfbox{ffuniversal-41.mps} \hskip 5mm
\epsfbox{ffuniversal-42.mps}
\end{center}
\caption{The decorated constraints forcing the structure of the tile $E\times F$.\label{fig-benum-x-cp}}
\end{figure}

It remains to analyze the tile $E\times F$.
Consider the decorated constraints given in Figure~\ref{fig-benum-x-cp}.
The first two constraints in Figure~\ref{fig-benum-x-cp} imply that for every $J\in\JJ_E$,
there exists an open interval $K(J)$ such that the following holds for almost every $(x,y)\in E\times F$:
$W(x,y)=1$ if $f_E(x)\in J$ and $f_F(y)\in K(J)$ for some $J\in\JJ_E$, and
$W(x,y)=0$ otherwise.
The third constraint implies that the intervals $K(J)$ and $K(J')$ are disjoint for $J\not=J'$, and
the fourth constraint yields that the measure of $K(J)$ is equal to $|J|$.
Finally, the last constraint implies that
if an interval $J_1\in\JJ_E$ precedes an interval $J_2\in\JJ_E$,
then $K(J_1)$ precedes the interval $K(J_2)$.
We conclude that $K(J)=J$ for every $J\in\JJ_E$.
Consequently, $W(x,y)=W_0(g(x),g(y))$ for almost every $(x,y)\in E\times F$.

\subsection{Referencing dyadic squares}
\label{sec-ref-dyadic-sq}
We now describe the tiles $E\times G$ and $F\times G$,
which allow referencing particular dyadic squares by the intervals from $\JJ_E$ and $\JJ_F$.
Formally, $W_0(\eta_E(x),\eta_G(y))=1$ for $x,y\in [0,1)$ if and only if
$x\in I_k$ and
$$\left\lfloor\frac{x-\min I_k}{|I_k|}\cdot 2^k\right\rfloor=
  \left\lfloor y\cdot 2^k\right\rfloor\;\mbox{,}$$
and it is equal to $0$ otherwise.
Similarly, $W_0(\eta_E(x),\eta_G(y))=1$ for $x,y\in [0,1)$ if and only if
$x\in I_k$ and
$$\left\lfloor\frac{x-\min I_k}{|I_k|}\cdot 2^{2k}\right\rfloor\;\equiv
  \left\lfloor y\cdot 2^k\right\rfloor\;(\mod 2^k)\;\mbox{,}$$
and it is equal to $0$ otherwise.
The tiles are depicted in Figure~\ref{fig-bin-tiles}.

\begin{figure}[htbp]
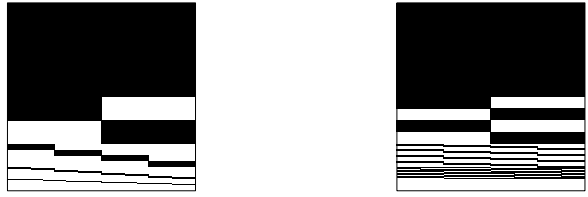

	\begin{center}
		\epsfbox{ffuniversal-309.mps} \hskip 25mm
		\epsfbox{ffuniversal-310.mps}
	\end{center}
	\caption{The tiles $E\times G$ and $F \times G$.\label{fig-bin-tiles}}
\end{figure}

\begin{figure}[htbp]
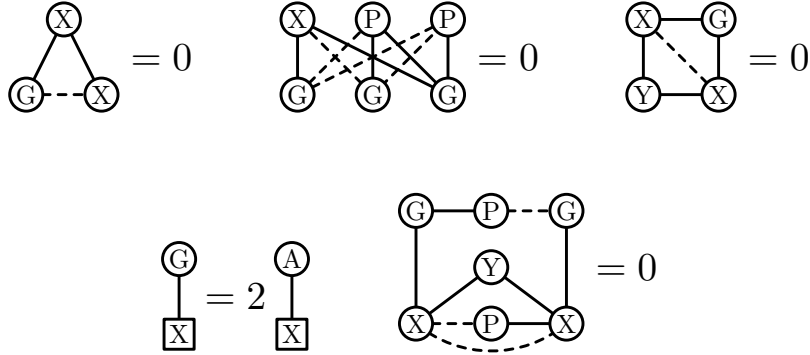

\begin{center}
\epsfbox{ffuniversal-45.mps} \hskip 10mm
\epsfbox{ffuniversal-44.mps} \hskip 10mm
\epsfbox{ffuniversal-46.mps} \vskip 10mm
\epsfbox{ffuniversal-43.mps} \hskip 10mm
\epsfbox{ffuniversal-47.mps} 
\end{center}
\caption{The decorated constraints forcing the structure of the tiles $E\times G$ and $F \times G$,
         where $(X, Y) \in \{(E, A), (F, E)\}$.\label{fig-bin}}
\end{figure}

Fix $X\in\{E,F\}$ and set $Y=A$ if $X=E$, and $Y=E$ if $X=F$.
Consider the decorated constraints given in Figure~\ref{fig-bin}.
The first two constraints in Figure~\ref{fig-bin} imply that for every $J\in\JJ_X$,
there exists an open interval $K^X(J)$ such that the following holds for almost every $(x,y)\in X\times G$:
$W(x,y)=1$ if $f_E(x)\in J$ and $f_F(y)\in K^X(J)$ for some $J\in\JJ_X$, and
$W(x,y)=0$ otherwise.

If $X=E$, the third constraint on the first line in Figure~\ref{fig-bin} implies that
if $J,J'\in\JJ_E$ and $J$ and $J'$ are subintervals of the same $0$-iterated binary interval,
then $K^E(J)$ and $K^E(J')$ are disjoint;
the second constraint on the second line implies that
if $J$ precedes $J'$ inside the same $0$-iterated binary interval, then $K^E(J)$ precedes $K^E(J')$.
Likewise, if $X=F$, the third constraint on the first line gives that
if $J,J'\in\JJ_F$ and $J$ and $J'$ are subintervals of the same interval contained in $\JJ_E$,
then $K^F(J)$ and $K^F(J')$ are disjoint, and
the second constraint on the second line gives that
if $J$ precedes $J'$ inside the same interval of $\JJ_E$, then $K^F(J)$ precedes $K^F(J')$.

Finally, the first constraint on the second line implies that $\deg_W^G(x)=2\deg_W^A(x)$
for almost every $x\in X$.
This implies that if $J$ is a subinterval of a $0$-iterated binary interval $J_k$, then $|K^X(J)|=2^{-k}$.
We conclude that $W(x,y)=W_0(g(x),g(y))$ for almost every $(x,y)\in X\times G$.

\subsection{Indexing dyadic squares}
\label{sec-forc-rec}

We now describe the tile $C\times F$,
which allows referencing particular dyadic squares by $2$-iterated binary intervals;
the tile is depicted in Figure~\ref{fig-inf-enum-tiles}.
Formally, $W_0(\eta_C(x),\eta_F(y))=1$ for $x,y\in [0,1)$ if and only if
$x\in I_{d,s,t}$, $y\in I_d$, $s<2^d$, $t<2^d$, and
$$\left\lfloor\frac{y-\min I_d}{|I_d|}\cdot 2^{2d}\right\rfloor=2^d\cdot s+t\mbox{,}$$
and it is equal to $0$ otherwise.

\begin{figure}[htbp]
	\begin{center}
		\epsfbox{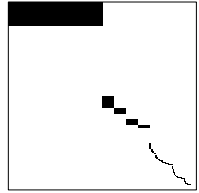}
	\end{center}
	\caption{The tile $C \times F$.\label{fig-inf-enum-tiles}}
\end{figure}

\begin{figure}[htbp]
\begin{center}
\epsfbox{ffuniversal-51.mps} \hskip 10mm
\epsfbox{ffuniversal-52.mps} \hskip 10mm
\epsfbox{ffuniversal-53.mps} \hskip 10mm
\epsfbox{ffuniversal-123.mps} \vskip 10mm
\epsfbox{ffuniversal-55.mps} \hskip 5mm
\epsfbox{ffuniversal-54.mps} \hskip 5mm
\epsfbox{ffuniversal-124} \hskip 5mm
\epsfbox{ffuniversal-125} \vskip 10mm
\epsfbox{ffuniversal-56.mps}
\end{center}
\caption{The decorated constraints forcing the structure of the tile $C \times F$. \label{fig-inf-enum}}
\end{figure}

Consider the constrains given in Figure~\ref{fig-inf-enum}.
The four constraints in the first line in Figure~\ref{fig-inf-enum} imply the following:
there exists a function $h:\NN_0^3\to\NN_0\cup\{\infty\}$ such that
$h(d,s,t)\in\{0,\ldots,2^{2d}-1\}\cup\{\infty\}$ and
the following holds for almost every $(x,y)\in C\times F$:
$W(x,y)=1$ if and only if $f_C(x)\in I_{d,s,t}$ and
$$\left\lfloor\frac{f_F(y)-\min I_d}{|I_d|}\cdot 2^{2d}\right\rfloor=h(d,s,t)\;\mbox{,}$$
$W(x,y)=0$ otherwise.
In particular, if $h(d,s,t)=\infty$, then $W(x,y)=0$ for almost every $x\in f_C^{-1}(I_{d,s,t})$ and $y\in F$.

The first constraint in the second line in Figure~\ref{fig-inf-enum} implies that
if $(d,s,t)\not=(d,s',t')$, then $h(d,s,t)\not=h(d,s',t')$ unless $h(d,s,t)=h(d,s',t')=\infty$.
The second constraint in the second line then implies that
if $h(d,s,t)$ and $h(d,s',t')$ are both different from $\infty$ and
$h(d,s,t)<h(d,s',t')$, then either $s=s'$ and $t<t'$ or $s<s'$.
The third constraint yields that
if $h(d,s,t)$ and $h(d,s',t')$ are both different from $\infty$ and $s\not=s'$ (the two $C$-roots
determine the values of the triples $(d,s,t)$ and $(d,s',t')$ and
their adjacencies to the $B$-root imply that $s\not=s'$),
then $\lfloor h(d,s,t)/2^d\rfloor\not=\lfloor h(d,s',t')/2^d\rfloor$.
Similarly, the last constraint yields that
if $h(d,s,t)$ and $h(d,s,t')$ are both different from $\infty$,
then $\lfloor h(d,s,t)/2^d\rfloor=\lfloor h(d,s,t')/2^d\rfloor$.
Consequently, for any $d$, there are at most $2^d$ values of $s$ such that $h(d,s,t)\not=\infty$ for some $t\in\NN_0$, and
for any $d$ and $s$, there are at most $2^d$ values of $t$ such that $h(d,s,t)\not=\infty$.

The density of the tile $C\times F$ is equal to the following:
\begin{equation}
\sum_{d=0}^{\infty}2^{-(3d+1)}\sum_{\substack{s,t\in\NN_0\\h(d,s,t)\not=\infty}} 2^{-d-s-t-3}
\label{eq-CN}
\end{equation}
Since for any $d$, there are at most $2^d$ values of $s$ such that $h(d,s,t)\not=\infty$ for some $t\in\NN_0$, and
for any $d$ and $s$, there are at most $2^d$ values of $t$ such that $h(d,s,t)\not=\infty$,
the inner sum in (\ref{eq-CN}) is at most
$$\sum_{s,t=0}^{2^d-1} 2^{-d-s-t-3}\;\mbox{.}$$
The constraint on the third line in Figure~\ref{fig-inf-enum} now yields that
$h(d,s,t)\not=\infty$ if and only if $s<2^d$ and $t<2^d$.
Since it holds that if $h(d,s,t)\not=\infty$, $h(d,s',t')\not=\infty$ and $h(d,s,t)<h(d,s',t')$,
then either $s=s'$ and $t<t'$ or $s<s'$,
it follows that $h(d,s,t)=2^d\cdot s+t$ for all $d$, $s<2^d$ and $t<2^d$.
It follows that $W(x,y)=W_0(g(x),g(y))$ for almost every $(x,y)\in C\times F$.

\subsection{Forcing densities}
\label{sec-rec-tile}

We now focus on the tile $G\times G$, which contains the graphon $W_F$ itself;
we define the value $W_0(\eta_G(x),\eta_G(y))$ to be equal to $W_F(x,y)$ for every $(x,y)\in [0,1)^2$.

\begin{figure}[htbp]
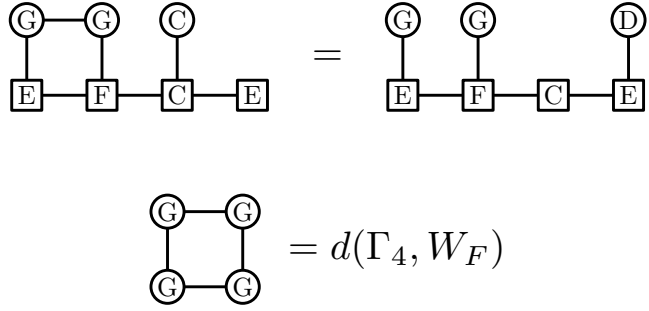

\begin{center}
\epsfbox{ffuniversal-102}\vspace{10mm}\\
\epsfbox{ffuniversal-304}
\end{center}
\caption{The decorated constraint forcing the $G\times G$ tile.\label{fig-rec-graphon}}
\end{figure}

Consider the first decorated constraint given in Figure~\ref{fig-rec-graphon}.
Almost every choice of the roots of the constraint satisfies the following:
if $x\in F$ is the $F$-root,
$f_F(x)\in I_d$, $d\in\NN_0$, and 
$$\left\lfloor\frac{f_F(x)-\min I_d}{|I_d|}\cdot 2^{2d}\right\rfloor=s\times 2^d+t\;\mbox{,}$$
where $s,t\in\{0,\ldots,2^d-1\}$,
then the left $E$-root $y\in E$ satisfies that $f_E(y)\in I_d$ and
$$\left\lfloor\frac{f_E(y)-\min I_d}{|I_d|}\cdot 2^d\right\rfloor=s\;\mbox{.}$$
Moreover,
the $C$-root $y'\in C$ and the right $E$-root $y''\in E$ satisfy that
$f_C(y')\in I_{d,s,t}$, and $f_E(y'')\in I_{d,s,t}$.
The left hand side of the density constraint is then equal to
$$2^{-d-s-t-3}\int_{f_G^{-1}\left(I^d(s)\right)\times f_G^{-1}\left(I^d(t)\right)} W(x,y)\dd x\dd y\;\mbox{,}$$
and the right hand side of the density constraint is equal to
$$2^{-2d}\cdot  
  \deg_W^D(y'')=2^{-2d}\cdot2^{-d-s-t-3}\cdot\delta(d,s,t)\;\mbox{.}$$
It follows that
\begin{equation}
2^{2d}\int_{f_G^{-1}\left(I^d(s)\right)\times f_G^{-1}\left(I^d(t)\right)} W(x,y)\dd x\dd y=
\delta(d,s,t)\;\mbox{.}\label{eq-square}
\end{equation}

Fix a measurable bijection $\psi:[0,1]\to G$ such that $|\psi^{-1}(X)|=|X|/|G|$ for every measurable $X\subseteq G$, and
define a graphon $W_G$ as $W_G(x,y)=W(\psi(x),\psi(y))$ and
a graphon $W'_F$ as $W'_F(x,y)=W_F(f_G(\psi(x)),f_G(\psi(y)))$.
Observe that $W'_F(x,y)=W_0(g(\psi(x)),g(\psi(y)))$ for almost every $(x,y)\in G\times G$.
Note that the second constraint in Figure~\ref{fig-rec-graphon} implies that $d(\Gamma_4,W_G)=d(\Gamma_4,W_F)$,
which is equal to $d(\Gamma_4,W'_F)$.
We now show that $W_G$ and $W'_F$ are equal almost everywhere.

Suppose that $\cut{W_G-W'_F}=\varepsilon>0$; note that $\varepsilon\le 1$.
For $d\in\NN_0$, define a graphon $W^d$ to be a step graphon with parts
$\psi^{-1}\left(f_G^{-1}\left(I^d(k)\right)\right)$, $k=0,\ldots,2^d-1$, such that
$$W^d(x,y)=\delta(d,s,t) \mbox{ for } x\in\psi^{-1}\left(f_G^{-1}\left(I^d(s)\right)\right)\mbox{ and }
                                         y\in\psi^{-1}\left(f_G^{-1}\left(I^d(t)\right)\right)\mbox{.}$$
The sequence $(W^d)_{d\in\NN_0}$ forms a martingale on $[0,1]^2$, and
Doob's Martingale Convergence Theorem implies that $W^d$ converges uniformly to $W'_F$.
Hence, there exists $d\in\NN_0$ such that $\cut{W'_F-W^d}\le\varepsilon^4/1800$.
Apply Proposition~\ref{prop-regularity} to the graphon $W_G$ and
the partition $\psi^{-1}\left(f_G^{-1}\left(I^d(k)\right)\right)$, $k\in\{0,\ldots,2^d-1\}$,
to obtain a step graphon $W'_G$ that refines $W^d$ and is $\varepsilon^4/1800$-close to $W_G$.
Consequently, we get $\cut{W'_G-W^d}\ge\varepsilon-\varepsilon^4/900\ge\varepsilon/2$,
which implies that 
\begin{equation}
d(\Gamma_4,W'_G)-d(\Gamma_4,W^d)\ge\varepsilon^4/128\label{eq-approx-C4}
\end{equation}
by Lemma~\ref{lm-C4}.
On the other hand, the choice of $W'_G$ and $W^d$ implies that
\begin{equation}
  \left|d(\Gamma_4,W_G)-d(\Gamma_4,W'_G)\right|\le\varepsilon^4/300\mbox{ and}
  \left|d(\Gamma_4,W'_F)-d(\Gamma_4,W^d)\right|\le\varepsilon^4/300\;\mbox{.}\label{eq-approx}
\end{equation}  
The inequalities (\ref{eq-approx-C4}) and (\ref{eq-approx}) now yield that $d(\Gamma_4,W'_F)>d(\Gamma_4,W_G)$.
However, this is impossible since $d(\Gamma_4,W'_F)=d(\Gamma_4,W_G)$.
Hence, the graphons $W_G$ and $W'_F$ are equal almost everywhere,
which implies that $W(x,y)$ and $W_0(g(x),g(y))$ are equal for almost every $(x,y)\in G\times G$.

\section{Cleaning up}
\label{sec-white}

We are close to finishing the description and the argument that the graphon $W_0$ is finitely forcible.
Let us start with the remaining tiles between the parts $A,\ldots,G$.
Fix $(X,Y)$ to be one of the pairs $(B,G)$, $(C,G)$, or $(D,G)$
and define $W_0(\eta_X(x),\eta_Y(y))=0$ for all $(x,y)\in [0,1)^2$.
Clearly, the first decorated constraint in Figure~\ref{fig-deg-empty} forces $W$
to be equal to zero for almost every $(x,y)\in X\times Y$.
Hence, we can conclude that $W(x,y)=W_0(g(x),g(y))$ for almost every pair $(x,y)\in (A\cup\cdots\cup G\cup P)^2$.

Similarly, we define $W_0(\eta_Q(x),\eta_Q(y))=1$ for all $(x,y)\in [0,1)^2$;
this is easy to force by the second constraint in Figure~\ref{fig-deg-empty}.
Hence, $W(x,y)=W_0(g(x),g(y))$ for almost every pair $(x,y)\in Q\times Q$.

\begin{figure}[htbp]
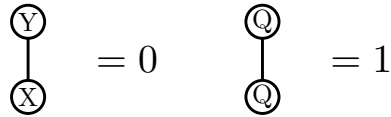

\begin{center}
\epsfbox{ffuniversal-133} \hskip 10mm
\epsfbox{ffuniversal-132.mps}
\end{center}
\caption{The decorated constraint forcing the tiles $X\times Y$,
         where $(X,Y)$ is one of the pairs $(B,G)$, $(C,G)$ and $(D,G)$, and
	 the decorated constraint forcing the tile $Q\times Q$.\label{fig-deg-empty}}
\end{figure}

\subsection{Degree balancing}
\label{sec-degrees}

We use the tiles $Q\times X$, where $X\in\{A,\ldots,G\}\cup\{P\}$, to guarantee that
$$\deg_{W_0}^{A_0\cup\cdots\cup G_0\cup P_0\cup Q_0}(x)=\frac{5}{13}$$
for every vertex $x\in A_0\cup\cdots\cup G_0\cup P_0$.
It may seem counterintuitive to force the degrees of the vertices in the parts $A,\ldots,G,P$ to be equal;
however, it is simpler to begin by enforcing the parts to be degree-regular (with the same degree) and
then enforce the different degrees of the parts.

First, note that $\deg_{W_0}^{A_0\cup\cdots\cup G_0\cup P_0}(x)\le\frac{5}{8}$ for every $x\in A_0\cup\cdots\cup G_0\cup P_0$;
this follows from that $|N_{W_0}^{A_0\cup\cdots\cup G_0\cup P_0}(x)|\le 5/14$ for every such $x$.
Let $\xi(x)=1-\frac{8}{5}\cdot\deg_{W_0}^{A_0\cup\cdots\cup G_0\cup P_0}(x)$ for every such $x$; note that $\xi(x)\in [0,1]$.
We define $W_0(x,y)=\xi(x)$ for every $x\in A_0\cup\cdots\cup G_0\cup P_0$ and $y\in Q_0$.
Further, we define $W_0(x,y)=1$ for all $(x,y)\in Q_0^2$.

\begin{figure}[htbp]
\begin{center}
\epsfbox{ffuniversal-129.mps} \hskip 10mm
\epsfbox{ffuniversal-128.mps}
\end{center}
\caption{The decorated constraints forcing the tiles $Q\times X$
         where $X\in\{A,\ldots,G\}\cup\{P\}$. \label{fig-deg-cons}}
\end{figure}

Fix $X\in\{A,\ldots,G\}\cup\{P\}$ and consider the decorated constraints given in Figure~\ref{fig-deg-cons}.
The first constraint implies that almost every $z$ and $z'$ from $Q$ satisfy that
$$\int\limits_{X}W(z,x)W(z',x)\dd x=\int\limits_{X_0}\xi(x)^2\dd x\;\mbox{.}$$
Lemma~\ref{lm-typ-pairs} implies that almost every $z$ from $Q$ satisfies that
$$\int\limits_{X}W(z,x)^2\dd x=\int\limits_{X_0}\xi(x)^2\dd x\;\mbox{.}$$
In particular, when $z$ is fixed and $W(z,x)$ is viewed as a function of $x$,
the $L_2$-norm of the function $W(z,x)$ for almost every $z\in Q$ is the same, and
the inner product of the functions $W(z,x)$ and $W(z',x)$ for almost every pair $z,z'\in Q$ is also the same and
equal to the $L_2$-norm.
Hence, the Cauchy-Schwarz Inequality yields that
there exists a function $h:X\to [0,1]$ such that
$W(z,x)=h(x)$ for almost every $x\in X$ and almost every $z\in Q$.
It follows that $W(x,z)=h(x)$ for almost every pair $(x,z)\in X\times Q$.

Since $\deg_W^{A\cup\cdots\cup G\cup P}(x)=\deg_{W_0}^{A_0\cup\cdots\cup G_0\cup P_0}(g(x))$ for almost every $x\in X$,
the second constraint in Figure~\ref{fig-deg-cons} implies that $h(x)=\xi(g(x))$ for almost every $x\in X$.
It follows that $W(x,y)=W_0(g(x),g(y))$ for almost every pair $(x,y)\in X\times Q$.
We now conclude that $W(x,y)=W_0(g(x),g(y))$ for almost every pair $(x,y)\in (A\cup\cdots\cup G\cup P\cup Q)^2$.

\subsection{Degree distinguishing}
\label{sec-degdist}

It remains to define and analyze the tiles $X\times R$, $X\in\{A,\ldots,G,P,Q,R\}$.
Fix $(X,k)$ to be one of the pairs $(A, 0), \ldots, (G,6), (P,7), (Q,8) (R, 9)$.
We define $W_0(x,y)=k/18$ for all $x\in X_0$ and $y\in R_0$.
It is easy to check that each vertex of $X_0$ has the same degree in $W_0$, and
this degree is equal to the one given in Table~\ref{tab-degrees}.
Also note that the common degree of the vertices of $Q_0$ is at least $5/14+8/252$.

\begin{figure}[htbp]
\begin{center}
\epsfbox{ffuniversal-130.mps} \hskip 10mm
\epsfbox{ffuniversal-131.mps}
\end{center}
\caption{The decorated constraints used to force the structure of the tiles $X\times R$ 
where $(X, k) \in \{(A, 0), \ldots, (G,6), (P,7), (Q,8), (R, 9)\}$.\label{fig-deg-dist}}
\end{figure}

Consider the two constraints given in Figure~\ref{fig-deg-dist}.
The first constraint implies that it holds for almost every $x\in X$ that
$$\frac{1}{|R|}\int\limits_{R}W(x,y)\dd y=\frac{k}{18}\;\mbox{,}$$
and the second constraint in Figure~\ref{fig-deg-dist} implies that 
it holds for almost every pair $(x,x')\in X^2$ that
$$\frac{1}{|R|}\int\limits_{R}W(x,y)W(x',y)\dd y=\left(\frac{k}{18}\right)^2\;\mbox{.}$$
We conclude using Lemma~\ref{lm-typ-pairs} that it holds that
$$\frac{1}{|R|}\int\limits_{R}W(x,y)^2\dd y=\left(\frac{k}{18}\right)^2$$
for almost every $x\in X$.
The Cauchy-Schwarz Inequality now yields that $W(x,y)=k/18$ for almost every pair $(x,y)\in X\times R$.
We can now conclude that $W(x,y)=W_0(g(x),g(y))$ for almost every $(x,y)\in X\times R$.

We have shown that if a graphon $W$ satisfies the presented decorated constraints,
then $W(x,y)=W_0(g(x),g(y))$ for almost every $(x,y)\in [0,1]^2$.
Since all the presented decorated constraints are satisfied by $W_0$ and
they can be turned into ordinary constraints by Lemma~\ref{lm-decorated},
the proof of Theorem~\ref{thm-main} is now finished.

\section{Conclusion}

The only constraints used to force the structure of the graphon $W_0$ that
depend on the graphon $W_F$ are the last constraint in Figure~\ref{fig:inputTile},
the last constraint in Figure~\ref{fig-rec-graphon} and the first constraint in Figure~\ref{fig-deg-cons}.
In each of the three constraints,
the structure of the graphon $W_F$ influences the numerical value of the right side of the constraint only.
Hence, Theorem~\ref{thm-main} holds in the following stronger form.

\begin{theorem}
\label{thm-main-stronger}
There exist graphs $H_1,\ldots,H_m$ with the following property.
For every graphon $W_F$,
there exist a graphon $W_0$ and reals $\delta_1,\ldots,\delta_m\in [0,1]$ such that
$W_F$ is a subgraphon of $W_0$ that is formed by a $1/14$ fraction of the vertices of $W_0$ and
the graphon $W_0$ is the only graphon $W$, up to a weak isomorphism, such that $d(H_i,W)=\delta_i$ for all $i\in [m]$.
\end{theorem}

The construction presented in the proof of Theorem~\ref{thm-main}
can be viewed as a map from the space of all graphons to the space of finitely forcible graphons.
The particular map implied by the proof of Theorem~\ref{thm-main} is not continuous with respect to the cut norm topology (and
we have not attempted to achieve this property).
However, the following weaker statement can be of possible use.
The statement easily follows from the proof of Theorem~\ref{thm-main}
since the $L_1$-distance between the functions defining the graphons $W_0$ and $W'_0$
is at most $\varepsilon$ for a suitable value of $k$.

\begin{proposition}
\label{prop-semicontinuous}
For every $\varepsilon>0$, there exists $k\in\NN$ such that the following holds.
If $W_F$ and $W'_F$ are two graphons such that the densities of their dyadic squares of sizes at least $2^{-k}$
agree up to the first $k$ bits after the decimal point in the standard binary representation,
then the cut distance between the finitely forcible graphons $W_0$ and $W'_0$
containing $W_F$ and $W'_F$, respectively, that are constructed in the proof of Theorem~\ref{thm-main},
is at most $\varepsilon$.
\end{proposition}

\section*{Acknowledgment}

The authors would like to thank Anton Bernshteyn, Andrzej Grzesik, Tom\'a\v s Kaiser and Jordan Venters
for stimulating discussions on the topics covered in this paper.
The authors would also like to thank the anonymous referee for many detailed comments that
helped to improve the presentation of the result and to clarify several arguments used in its proof.

\newcommand{\advances}{Adv. Math. }
\newcommand{\annals}{Ann. of Math. }
\newcommand{\cpc}{Combin. Probab. Comput. }
\newcommand{\dcg}{Discrete Comput. Geom. }
\newcommand{\discrete}{Discrete Math. }
\newcommand{\eur}{European J.~Combin. }
\newcommand{\gfa}{Geom. Funct. Anal. }
\newcommand{\jcta}{J.~Combin. Theory Ser.~A }
\newcommand{\jctb}{J.~Combin. Theory Ser.~B }
\newcommand{\rsa}{Random Structures Algorithms }
\newcommand{\sidma}{SIAM J.~Discrete Math. }
\newcommand{\tcs}{Theoret. Comput. Sci. }

\end{document}